\RequirePackage{ifpdf}
\ifpdf 
\documentclass[pdftex]{sigma}
\else
\documentclass{sigma}
\fi

\usepackage[s]{optional} 

\usepackage{bbm}
\usepackage{tikz-cd}
\usetikzlibrary{matrix,arrows, decorations,decorations.markings,calc,positioning}
\usepackage{float}
\usepackage{setspace}

\newcommand\picskip{\smallskip}

\newcommand{\setcondsbig}[2]{\ensuremath{\big\{#1\colon #2 \big\}}}

\newcommand{\quotstack}[2]{\ensuremath{\big[#1\big/\,#2\big]}}

\newcommand{\quotstackBig}[2]{\ensuremath{\big[#1\big/#2\big]}}
\newlength\myWindowLength
\newcommand{\WindowGens}[3]{#1\colon l\in \left[0,#2\right),\, m\in \left[#3\right)}
\newcommand{\WindowBox}[3]{\settowidth{\myWindowLength}{$\WindowGens{#1}{#2}{#3}$}
\begin{minipage}[c][15pt][c]{\myWindowLength}{$\WindowGens{#1}{#2}{#3}$}\end{minipage}}
\newcommand{\WindowSet}[3]{\big\{\WindowBox{#1}{#2}{#3} \big\}}

\newcommand{\WindowGensNew}[3]{#1\colon l\in \left[0,#2\right],\, m\in \left(#3\right]}
\newcommand{\WindowBoxNew}[3]{\settowidth{\myWindowLength}{$\WindowGensNew{#1}{#2}{#3}$}
\begin{minipage}[c][15pt][c]{\myWindowLength}{$\WindowGensNew{#1}{#2}{#3}$}\end{minipage}}
\newcommand{\WindowSetNew}[3]{\big\{ \WindowBoxNew{#1}{#2}{#3} \big\}}

\newcommand\xyhook{\ar@{^{(}->}}

\newcommand\C{\mathbb C}

\newcommand\id{\operatorname{id}}

\newcommand\into{\hookrightarrow}

\DeclareMathOperator \Hom {Hom}
\renewcommand\hom{\mathcal{H}{\rm om} }
\DeclareMathOperator\Ext{Ext}
\newcommand\RDerived{\mathrm{R}}
\DeclareMathOperator \RHom {RHom}
\DeclareMathOperator \RsHom {R\mathcal{H}om}
\renewcommand\Im{\operatorname{Im}}
\renewcommand\P{\mathbb P}
\newcommand{\Tw}{\operatorname{Tw}}
\newcommand{\Tr}{\operatorname{Tr}}
\newcommand{\Gr}{\mathrm{G}}
\newcommand{\Pf}{\mathrm{P}}

\newcommand\cL{\mathcal L}
\newcommand\cM{\mathcal M}

\newcommand\rk{\operatorname{rk}}

\newcommand{\cA}{\mathcal{A}}
\newcommand{\cB}{\mathcal{B}}

\newcommand{\cE }{\mathcal{E}}
\newcommand{\cF}{\mathcal{F}}

\newcommand{\cO}{\mathcal{O}}

\newcommand{\cG}{\mathcal{G}}
\newcommand{\cW}{\mathbb{W}}

\newcommand{\mfG}{\mathcal{G}}
\newcommand{\mfP}{\mathcal{P}}

\newcommand{\mfX}{\mathfrak X}

\newcommand{\Sym}{\operatorname{Sym}}

\newcommand{\Serre}{\mathbb{S}}

\DeclareMathOperator \SL {SL}
\DeclareMathOperator \GL {GL}

\newcommand \Caldararu {C\u{a}ld\u{a}\-raru}

\newlength\tempWidth
\newcommand\InSpaceOf[2]{
 \settowidth{\tempWidth}{$#1$}\phantom{#1}\hspace{-\tempWidth}{#2}}

\newcommand\Br{\mathrm{Br}}

\newcommand\Cone{\mathrm{Cone}}
\newcommand\D{\mathrm{D^b}}

\newcommand\twoform[1]{\omega_{#1}}

\newcommand\smallGLSrep[2]{
\def\exp{#1} \def\zero{0}
\def\twist{#2} \def\zero{0}
\smallSymS{#1} \ifx\twist\zero \phantom{(} \else {({#2})}\fi
}

\newcommand\smallGLSrepnew[2]{
\def\exp{#1} \def\zero{0}
\def\twist{#2} \def\zero{0}
\smallSymSnew{#1} \ifx\twist\zero \phantom{(} \else {({#2})} \fi
}

\newcommand\smallGLSrepnewabbr[2]{
S_{#1,#2}
}

\newcommand\smallSymS[1]{
\def\exp{#1} \def\zero{0} \def\one{1}
\ifx\exp\zero \cO \else \ifx\exp\one S^{\vee} \else {S}^{{#1} \vee} \fi \fi
}

\newcommand\smallSymSnew[1]{
\def\exp{#1} \def\zero{0} \def\one{1}
\ifx\exp\zero \cO \else \ifx\exp\one S \else S^{{#1}} \fi \fi
}

\newcounter{keyenvcount}

\numberwithin{equation}{section}

\newtheorem{thm}{Theorem}[section]

\newtheorem{prop}[thm]{Proposition}
\newtheorem{lem}[thm]{Lemma}
\newtheorem{cor}[thm]{Corollary}
\newtheorem{col}{Collection}

\theoremstyle{definition}
\newtheorem{defn}[thm]{Definition}

\newtheorem{notn}[thm]{Notation}
\newtheorem{eg}[thm]{Example}
\newtheorem{rem}[thm]{Remark}

\newcommand{\marginparstretch}{0.6}
\let\oldmarginpar\marginpar
\renewcommand\marginpar[1]{\-\oldmarginpar[\framebox{\setstretch{\marginparstretch}\begin{minipage}{\marginparwidth}{\raggedleft\scriptsize #1}\end{minipage}}]{\framebox{\setstretch{\marginparstretch}\begin{minipage}{\marginparwidth}{\raggedright\scriptsize #1}\end{minipage}}}}

\newcommand\segment[3]{to [out=#1,in=#3+180] (#2)}
\newcommand\circsegment[3]{arc (#1+90:#3+90:#2 and #2)}

\def\colorlinebundletop{green!60!black}
\def\colorlinebundlebottom{red}

\def\colortwist{black}
\def\colorfront{black!20}
\def\colorback{black!10}
\def\monodromystyle{semithick}

\newcommand\skmsPic[2]{
\def\type{#1}
\def\labels{#2}
\def\typeOne{1}
\def\typeTwo{2}
\def\labelsMod{1}

\def\vertspherescale{0.27}
\def\spherescale{1.15 }
\begin{tikzpicture}[>=stealth,scale=2.5]

\node (Mlabel) at (135:1.25*\spherescale) {$ \cM$};
\draw[thick] ([shift=(-84:\spherescale)]0,0) arc (-84:84:\spherescale)
 [bend left] to (96:\spherescale)
 arc (96:264:\spherescale)
 [bend left] to cycle;
\draw[\colorback] (\spherescale,0) arc (0:180:\spherescale and \vertspherescale);
\draw[\colorfront] (\spherescale,0) arc (0:-180:\spherescale and \vertspherescale)
coordinate[pos=0.60,above=0.35] (C) coordinate[pos=0.47,above=0.0] (D) coordinate[pos=0.35,above=0.35] (E);
\filldraw[fill=white,draw=black] (C) circle (1.5pt);
\filldraw[fill=white,draw=black] (D) circle (1.5pt);
\filldraw[fill=white,draw=black] (E) circle (1.5pt);
\ifx\type\typeOne
\node (Clabel) at (C) [below=0.25] {\ifx\labels\labelsMod $\scriptstyle \Tw^2 $ \else $\scriptstyle \operatorname{Sym}^2 \!S$ \fi};
\node (Dlabel) at (D) [below=0.25] {\ifx\labels\labelsMod $\scriptstyle \Tw^1 $ \else $\scriptstyle S$ \fi};
\node (Elabel) at (E) [below=0.25] {\ifx\labels\labelsMod $\scriptstyle \Tw^0 $ \else $\scriptstyle \mathcal{O}$\fi}; \fi
\node (p) at (-0.4*\spherescale,0.6*\spherescale) {};
\node (plabel) at (p) [above left=-0.05] {$\scriptstyle m_{\Gr}$};

\ifx\type\typeTwo
\node (q) at (-0.35*\spherescale,-0.7*\spherescale) {};
\node (qlabel) at (q) [below left=-0.08] {$\scriptstyle m_{\Pf}$};
\fi
\node[color=\colorlinebundletop] at (0.25*\spherescale,0.78*\spherescale) {$\scriptstyle \otimes\mathcal{O}_{Y_\Gr}\!(1)$};
\draw[\monodromystyle,,color=\colorlinebundletop,bend left,looseness=0.5]
(p) to ([shift=(100:\spherescale)]0,0);
\draw[\monodromystyle,color=\colorlinebundletop,bend right,densely dotted] ([shift=(100:\spherescale)]0,0) to ([shift=(80:\spherescale)]0,0);
\draw[\monodromystyle,color=\colorlinebundletop, looseness=1.0,decoration={
markings,
mark=at position 0.5 with {\arrow{>}}},postaction=decorate]
([shift=(79:\spherescale)]0,0)
\segment{-130}{[shift=(90:0.85*\spherescale)]0,0}{-160}
\segment{-160}{p}{-115};

\ifx\type\typeOne
\draw[\monodromystyle,color=\colortwist, looseness=0.8, decoration={markings,
 mark=at position 0.54 with {\arrow[rotate=15]{>}}}, postaction=decorate] ($(p)+(100:0.2pt)$)
 \segment{10}{$(E)+(45:3.5pt)$}{-45}
 \circsegment{-45}{3.5pt}{-260}
 \segment{-260}{$(p)+(100:-0.2pt)$}{190};

\draw[\monodromystyle, color=\colortwist, looseness=0.8, decoration={markings,
 mark=at position 0.54 with {\arrow[rotate=15]{>}}}, postaction=decorate] ($(p)+(80:0.2pt)$)
 \segment{-10}{$(D)+(25:3.5pt)$}{-65}
 \circsegment{-65}{3.5pt}{-285}
 \segment{-285}{$(p)+(80:-0.2pt)$}{170};

\draw[\monodromystyle, color=\colortwist, looseness=0.8, decoration={markings,
 mark=at position 0.54 with {\arrow[rotate=15]{>}}}, postaction=decorate] ($(p)+(0:0.2pt)$)
 \segment{-90}{$(C)+(5:3.5pt)$}{-85}
 \circsegment{-85}{3.5pt}{-305}
 \segment{-305}{$(p)+(0:-0.2pt)$}{90};
\fi

\ifx\type\typeTwo

\draw[\monodromystyle,looseness=0.5,decoration={
 markings,
 mark=at position 0.8 with {\arrow[rotate=0]{>}}},postaction=decorate] (p) \segment{-130}{q}{-40};
\draw[\monodromystyle,looseness=0.7,decoration={
 markings,
 mark=at position 0.8 with {\arrow[rotate=0]{>}}},postaction=decorate] (p) \segment{-30}{q}{-150};
\draw[\monodromystyle,looseness=1.5,decoration={
 markings,
 mark=at position 0.75 with {\arrow[rotate=0]{>}}},postaction=decorate] (p) \segment{-10}{q}{-170};
\draw[\monodromystyle,looseness=2.5,decoration={
 markings,
 mark=at position 0.7 with {\arrow[rotate=0]{>}}},postaction=decorate] (p) \segment{0}{q}{-180};
\node (path0) at (0.7,-0.55) {$\scriptstyle \Psi^0$};
\node (path1) at (0.32,-0.48) {$\scriptstyle \Psi^1$};
\node (path2) at (-0.29,-0.45) {$\scriptstyle \Psi^2$};
\node (path3) at (-0.75,-0.47) {$\scriptstyle \Psi^3$};
\fi

\ifx\type\typeOne
\def\circRadius{26pt}
\def\circRadiusShift{26pt}
\coordinate (circCentre) at ($(p)+(-50:\circRadius)$);
\coordinate (circCentreShift) at ($(p)+(-56:\circRadius)$);
\node[color=\colorlinebundlebottom] at (0.23*\spherescale,-0.65*\spherescale) {$\scriptstyle \Psi^{ -1}(\otimes\mathcal{O}_{Y_\Pf}\!(1))\Psi$};
\draw[\monodromystyle,color=\colorlinebundlebottom, looseness=1.2,decoration={
 markings,
 mark=at position 0.8 with {\arrow[rotate=0]{>}}},postaction=decorate]
($(p)+(-55:0.2pt)$)
\segment{-145}{$(circCentre)+(150:\circRadius)$}{-120}
\circsegment{60}{\circRadius}{150}
\segment{-30}{[shift=(-80:\spherescale)]0,0}{-50};
\draw[\monodromystyle,color=\colorlinebundlebottom,looseness=1.2]
($(p)+(-55:-0.2pt)$)
\segment{-145}{$(circCentreShift)+(150:\circRadiusShift)$}{-120}
\circsegment{60}{\circRadiusShift}{150}
\segment{-30}{[shift=(-100:\spherescale)]0,0}{-50};
\fi
\ifx\type\typeTwo
\node[color=\colorlinebundlebottom] at (0.23*\spherescale,-0.78*\spherescale) {$\scriptstyle \otimes\mathcal{O}_{Y_\Pf}\!(1)$};
\draw[\monodromystyle,color=\colorlinebundlebottom, looseness=1.2,decoration={
 markings,
 mark=at position 0.5 with {\arrow[rotate=5]{>}}},postaction=decorate]
(q)
\segment{-30}{[shift=(-90:0.85*\spherescale)]0,0}{-30}
\segment{-30}{[shift=(-80:\spherescale)]0,0}{-50};
\draw[\monodromystyle,color=\colorlinebundlebottom,bend right,looseness=0.5]
(q) to ([shift=(-100:\spherescale)]0,0);
\fi
\draw[\monodromystyle,color=\colorlinebundlebottom,bend left,densely dotted] ([shift=(-100:\spherescale)]0,0) to ([shift=(-80:\spherescale)]0,0);

\filldraw[color=white] (p) circle (2pt);
\filldraw (p) circle (1pt);
\ifx\type\typeTwo
\filldraw[color=white] (q) circle (2pt);
\filldraw (q) circle (1pt);
\fi
\draw[thick] ([shift=(-84:\spherescale)]0,0) arc (-84:84:\spherescale)
[bend left] to (96:\spherescale)
arc (96:264:\spherescale)
[bend left] to cycle;
\end{tikzpicture}
}

\newcommand\labelpos{1.3}
\newcommand\labelposrelax{0.2}

\newcommand\labeladjust{0.05}

\newcommand\monodScale{1.2}

\newcommand{\monodromiesStandardFlop}[4]{
\begin{tikzpicture}[scale=\monodScale,node distance=1cm, auto, line width=0.5pt]

\node (neg) at (-\labelpos-\labelposrelax,-\labeladjust) {#1};

\draw[<-, rounded corners=10pt] (-0.6,-0.1) -- (0.4,-0.35) -- (0.4,0.35) -- (-0.6,0.1);
\draw[<-, rounded corners=10pt] (-0.85,-0.4) -- (0,-1.5) -- (0.5,-1) -- (-0.7,-0.3) ;
\draw[->, rounded corners=10pt] (-0.85,0.4) -- (0,1.5) -- (0.5,1) -- (-0.7,0.3) ;

\draw (0,-1) circle (.4ex);
\draw (0,0) circle (.4ex);
\draw (0,+1) circle (.4ex);

\node at (0.8,-1.2) {\scriptsize #2};
\node at (0.8,0) {\scriptsize #3};
\node at (0.7,1.2) {\scriptsize #4};

\end{tikzpicture}
}

\newcommand{\halfMonodromiesStandardFlop}[3]{
\begin{tikzpicture}[scale=\monodScale, node distance=1cm, auto, line width=0.5pt]

\node (pos) at (-\labelpos-\labelposrelax,-\labeladjust) {#2};
\node (neg) at (+\labelpos+\labelposrelax,-\labeladjust) {#1};

\draw[->, rounded corners=30pt] (-1.1,-0.4) -- (0,-1.8) -- (1.1,-0.4) ;
\node (b) at (1.0,-1.2) {\scriptsize $#3^3$};
\draw[->, rounded corners=15pt] (-0.85,-0.3) -- (0,-0.8) -- (0.85,-0.3) ;
\node (b) at (0.4,-0.8) {\scriptsize $#3^2$};
\draw[->, rounded corners=15pt] (-0.85,0.3) -- (0,0.7) -- (0.85,0.3) ;
\node (b) at (0.4,0.8) {\scriptsize $#3^1$};
\draw[->, rounded corners=30pt] (-1.1,0.4) -- (0,1.8) -- (1.1,0.4) ;
\node (b) at (0.8,1.3) {\scriptsize $#3^0$};

\draw (0,-1) circle (.4ex);
\draw (0,0) circle (.4ex);
\draw (0,+1) circle (.4ex);

\end{tikzpicture}
}

\newcommand{\halfFullMonodromiesStandardFlop}[3]{
\begin{tikzpicture}[scale=\monodScale, node distance=1cm, auto, line width=0.5pt]

\node (pos) at (-\labelpos-\labelposrelax,-\labeladjust) {#2};
\node (neg) at (+\labelpos+\labelposrelax,-\labeladjust) {#1};

\draw[->, rounded corners=30pt] (-1.1,-0.4) -- (0,-1.8) -- (1.1,-0.4) ;
\node (b) at (1.0,-1.2) {\scriptsize $#3^3$};
\draw[->, rounded corners=15pt] (-0.85,-0.3) -- (0,-0.8) -- (0.85,-0.3) ;
\node (b) at (0.4,-0.8) {\scriptsize $#3^2$};
\draw[->, rounded corners=15pt] (-0.85,0.3) -- (0,0.7) -- (0.85,0.3) ;
\node (b) at (0.4,0.8) {\scriptsize $#3^1$};
\draw[->, rounded corners=30pt] (-1.1,0.4) -- (0,1.8) -- (1.1,0.4) ;
\node (b) at (0.8,1.3) {\scriptsize $#3^0$};

\draw[<-, rounded corners=10pt] (2.54,-0.1) -- (3.4,-0.35) -- (3.4,0.35) -- (2.4,0.1);
\node at (4.1,0) {\scriptsize $\otimes\cO_{Y_{\Pf}}(1)$};
\draw (3,0) circle (.4ex);

\draw[->, rounded corners=10pt] (-2.4,-0.1) -- (-3.4,-0.35) -- (-3.4,0.35) -- (-2.4,0.1);
\node at (-4.1,0) {\scriptsize $\otimes\cO_{Y_{\Gr}}(1)$};
\draw (-3,0) circle (.4ex);

\draw (0,-1) circle (.4ex);
\draw (0,0) circle (.4ex);
\draw (0,+1) circle (.4ex);

\end{tikzpicture}
}

\begin{document}

\newcommand{\arXivNumber}{2009.12630}

\renewcommand{\thefootnote}{}

\renewcommand{\PaperNumber}{028}

\FirstPageHeading

\ShortArticleName{Stringy K\"ahler Moduli for the Pfaffian--Grassmannian Correspondence}

\ArticleName{Stringy K\"ahler Moduli for the Pfaffian--Grassmannian Correspondence\footnote{This paper is a~contribution to the Special Issue on Primitive Forms and Related Topics in honor of~Kyoji Saito for his 77th birthday. The full collection is available at \href{https://www.emis.de/journals/SIGMA/Saito.html}{https://www.emis.de/journals/SIGMA/Saito.html}}}

\Author{Will DONOVAN}

\AuthorNameForHeading{W.~Donovan}

\Address{Yau Mathematical Sciences Center, Tsinghua University,\\ Haidian District, Beijing 100084, China}
\Email{\href{mailto:donovan@mail.tsinghua.edu.cn}{donovan@mail.tsinghua.edu.cn}}

\ArticleDates{Received September 29, 2020, in final form March 10, 2021; Published online March 24, 2021}

\Abstract{The Pfaffian--Grassmannian correspondence relates certain pairs of derived equi\-valent non-birational Calabi--Yau 3-folds. Given such a pair, I construct a set of deri\-ved equivalences corresponding to mutations of an exceptional collection on the relevant Grassmannian, and give a mirror symmetry interpretation, following a physical analysis of~Eager, Hori, Knapp, and Romo.}


\Classification{14F08; 14J32; 14M15; 18G80; 81T30}

\renewcommand{\thefootnote}{\arabic{footnote}}
\setcounter{footnote}{0}

\section{Introduction}

Birational Calabi--Yau $3$-folds are known to have equivalent derived categories~\cite{BriFlop}. There also exist pairs of Calabi--Yau $3$-folds which are not birational but may be proved to be derived equivalent. A much-studied class of examples comes from the ``Pfaffian--Grassmannian'' correspondence, concerning pairs of $3$-folds arising as linear sections of the Grassmannian $\Gr(2,7)$ and~its projective dual Pfaffian. R\o{}dland conjectured that such pairs share a mirror~\cite{Rodland}, leading to~an~expectation that they are derived equivalent: this was proved by Borisov and \Caldararu~\cite{BorCal}, and Kuznetsov \cite{KuznetsovHPDlines}.

Meanwhile, Hori and Tong \cite{HT} gave a physical explanation of how such pairs of $3$-folds arise from the same gauged linear $\sigma$-model.
In work of Addington, Segal, and the author~\cite{ADS}, a partial mathematical interpretation of this was given by constructing a particular derived equivalence for each pair using categories of matrix factorizations.

According to mirror symmetry, the derived symmetries of a variety may be determined by~mono\-dromy on a stringy K\"ahler moduli space (SKMS). Hori and Tong described this space for such $3$-folds: implicit in this was a prediction of further equivalences, corresponding to ``grade restriction windows'', see~\cite{HHP,Segal}. Later physics work made these windows explicit, and argued that differences between equivalences are given by spherical twists: see Eager, Hori, Knapp, and Romo~\cite{EHKR}, and Hori~\cite[end of Section~5]{HoriKyoto}.

In this paper, I interpret this physics work by constructing, for each pair of $3$-folds coming from the Pfaffian--Grassmannian correspondence, a set of equivalences corresponding to the ``windows'' above by extending the methods of~\cite{ADS}. I then show that these equivalences, along with appropriate spherical twists, may be organized into an action of the fundamental group of~the relevant SKMS.

\subsection{Calabi--Yau pairs}\label{sec.threefolds} I recall the construction of the $3$-folds $Y_{\Gr}$ and $Y_{\Pf}$. Start with a $7$-dimensional vector space $V$, and consider the following.
\begin{itemize}\itemsep=0pt
\item The Grassmannian of $2$-planes in its Pl\"ucker embedding
\begin{equation*}
\Gr(2,V) \subset \P\big({\wedge}^2 V\big).
\end{equation*}
\item The Pfaffian of 2-forms on $V$ of rank at most $4$, denoted
\begin{equation*}
\Pf(4,V) \subset \P\big({\wedge}^2 V^\vee\big).
\end{equation*}
\end{itemize}
These varieties are projectively dual. The Pfaffian $\Pf(4,V)$ is singular along the locus of forms of~rank at most $2$, but taking sufficiently generic hyperplane sections yields a smooth Calabi--Yau $3$-fold~$Y_{\Pf}$. Taking a dimension 7 subspace $\Pi \subset \wedge^2 V^\vee$ and its annihilator $\Pi^\circ \subset \wedge^2 V$, we then obtain smooth Calabi--Yau $3$-folds as follows:\footnote{For details, see \cite{BorCal}. In fact, if $\Pi$ is chosen such that $Y_\Pf$ is a smooth $3$-fold, then the same is true of $Y_\Gr$~\cite[Corollary~2.3]{BorCal}.}
\begin{gather*}
Y_{\Gr} = \Gr(2,V) \cap \P\Pi^\circ,\\
Y_{\Pf} = \Pf(4,V) \cap \P\Pi.
\end{gather*}

\subsection{Equivalences}

For each pair $Y_{\Gr}$ and $Y_{\Pf}$, I construct a set of derived equivalences depending on a discrete parameter given as follows:
\[
\underline{m}=(m_0,m_1,m_2)\in\mathbb{Z}^3\qquad \text{such that}\quad m_l \leq m_{l+1} \leq m_l + 1.
\]
Each choice of $\underline{m}$ gives an exceptional collection on $\Gr(2,V)$ by successive mutations of a collection due to Kuznetsov (Proposition~\ref{prop.exc}). This collection is determined by a Lefschetz block
\[
\cO_{\Gr}(m_0), \qquad
S(m_1), \qquad
\Sym^2 S(m_2),
\]
where $S$ is the rank $2$ tautological subspace bundle on $\Gr(2,V)$.

Now, by extending the construction of \cite{ADS}, I prove the following.

\begin{thm}[Theorem~\ref{thm.equiv}]\label{mainthm.equiv}
For each $\underline{m}\in\mathbb{Z}^3$ as above, an equivalence
\begin{equation*}
\Psi^{\underline{m}}\colon\ \D(Y_{\Gr}) \overset{\sim}{\longrightarrow} \D(Y_{\Pf})
\end{equation*}
may be constructed using the exceptional collection on $\Gr(2,V)$ given by~$\underline{m}$.
\end{thm}

Section~\ref{sec.struct} outlines how $\Psi^{\underline{m}}$ is constructed: it is a ``window equivalence'', where generators are obtained from the collection given by $\underline{m}$. The equivalence obtained in~\cite{ADS} corresponds to $\underline{m}=(6,7,8)$ (Remark~\ref{rem.adsgens}).

\subsection{Physics} The SKMS in our case is described in the physics literature as a sphere with five punctures: see~\cite[Fig.~1]{HT} and~\cite[Section~4]{EHKR}. Our $3$-folds correspond to ``large radius limits'' near two of these punctures, shown below as the poles. In this picture, homotopy classes of paths between large radius limits are expected to correspond to derived equivalences, given by ``grade restriction windows''. Such equivalences are supplied by Theorem~\ref{mainthm.equiv}. Monodromy around the other three punctures are then expected to correspond to spherical twists \cite[end of Section~5]{HoriKyoto}. Theorems~\ref{keythm.groupoid} and~\ref{keythm.twists} below confirm these expectations from the physics literature.

\subsection{Groupoid action}\label{sec.introactions}
I show that the fundamental groupoid of the SKMS acts on the derived categories $\D(Y_{\Gr})$ and~$\D(Y_{\Pf})$.

Let $\cM = S^2 - \{\text{$5$ points}\}$ where two of the punctures are the poles, and basepoints $m_{\Gr}$ and $m_{\Pf}$ are chosen near them. Following the physics analysis~\cite[Section~4]{EHKR}, I choose a finite subset~$\Psi^k$, $0\leq k\leq 3$, of the equivalences~$\Psi^{\underline{m}}$ of Theorem~\ref{mainthm.equiv} as follows:
\begin{equation*}
\Psi^k = \Psi^{\underline{m}},\qquad \text{where}\quad
m_l = \begin{cases}
-1, & l<k, \\ \hphantom{-}0, & l\geq k.
\end{cases}
\end{equation*}
The corresponding exceptional collections are illustrated in Section~\ref{sec.cols} as \mbox{Collections~\ref{cola}--\ref{cold}}. Combining the construction of Theorem~\ref{mainthm.equiv} with standard techniques for manipulating window equi\-va\-lences, I obtain the following.
\begin{thm}[Theorem~\ref{main thm}]\label{keythm.groupoid}
There is an action of the fundamental groupoid $\pi_1(\cM,\{m_{\Gr},m_{\Pf}\})$ on $\D(Y_{\Gr})$ and $\D(Y_{\Pf})$, given by the following diagram:
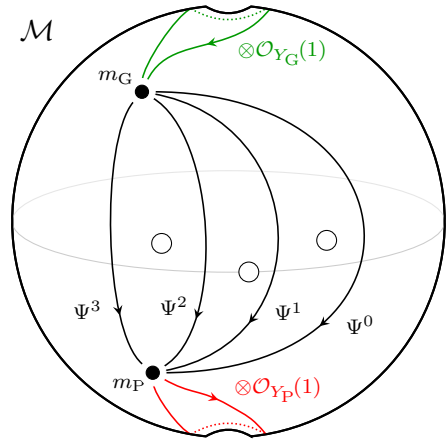
\picskip
\begin{center}
\skmsPic{2}{0}
\end{center}
\end{thm}

\subsection{Mutations and group action} The exceptional collections for the~$\Psi^k$ are related by mutations of exceptional objects as follows:
\[
\begin{tikzpicture}[xscale=2.5]
\node (a) at (-0.5,0) {Collection \ref{cola}};
\node (b) at (1,0) {\ref{colb}};
\node (c) at (2,0) {\ref{colc}};
\node (d) at (3,0) {\ref{cold}.};
\draw[->] (a) to node[above]{$\cO_{\Gr}$} (b);
\draw[->] (b) to node[above]{$S_{\phantom{\Gr}}$} (c);
\draw[->] (c) to node[above]{$\Sym^2 S$} (d);
\end{tikzpicture}
\]
The restrictions of these three exceptional objects to $Y_{\Gr}$ are the three spherical objects below (Proposition~\ref{prop.sphobj}):
\begin{equation}\label{eq.sphs}
\cO_{Y_{\Gr}}, \qquad S_{Y_{\Gr}}, \qquad \Sym^2 S_{Y_{\Gr}}.
\end{equation}

\begin{rem}
This phenomenon is analogous to exceptional objects restricting to spherical objects on an anticanonical divisor~\cite[Example~3.14(c)]{ST}, though here the codimension of $Y_{\Gr}$ is~$7$.
\end{rem}

I then show (Proposition~\ref{prop.twists}) that the ``differences'' between equivalences \[\big(\Psi^{j+1}\big)^{-1} \circ \Psi^j\]
 are spherical twists around the objects~\eqref{eq.sphs}. Combining this with Theorem~\ref{keythm.groupoid}, I deduce the following.

\begin{thm}[Theorem~\ref{thm.twists}]\label{keythm.twists}
There is an action of the fundamental group $\pi_1(\cM,m_{\Gr})$ on $\D(Y_{\Gr})$ given by the following diagram:
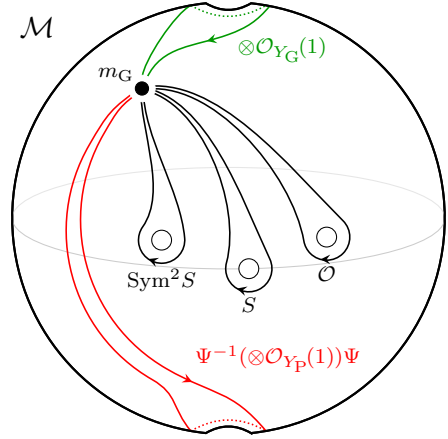
\picskip
\begin{center}
\skmsPic{1}{0}
\end{center}
\picskip
Here we let $\Psi=\Psi^3$, and the loops around equatorial holes indicate spherical twists around the objects~\eqref{eq.sphs}.
\end{thm}

It would be interesting to recover $\cM$, and the above actions, from an analysis of Bridgeland stability conditions.

\subsection{Contents}
Section~\ref{sec.cols} explains the exceptional collections in Theorem~\ref{mainthm.equiv}, Section~\ref{sec.struct} outlines the structure of the proof, and then Sections~\ref{sec.equivLG} and~\ref{sec.relating} give the details. Section~\ref{sec groupoid} proves Theorem~\ref {keythm.groupoid}, and Section~\ref{sec twists} proves Theorem~\ref {keythm.twists}.

\section{Exceptional collections} \label{sec.cols}

In this section, I explain the exceptional collections in Theorem~\ref{mainthm.equiv}, and convenient ways to visualize them.

The construction in \cite{ADS} used the following full exceptional collection on the Grassmannian $\Gr(2,V)$ for a 7-dimensional vector space $V$~\cite[Theorem~4.1]{KuznetsovECs}, where $S$ denotes the rank $2$ tautological subspace bundle.
\begin{equation}\label{eq:rectangularcollectionnew}
\WindowSet{\Sym^l S^\vee \otimes (\det S^\vee)^m_{\InSpaceOf{,}{}}}{3}{0,7}.
\end{equation}
I depict the collection \eqref{eq:rectangularcollectionnew} as follows, where as usual we write $\cO(1)$ for $\det S^\vee$, and for convenience in this section put $S^2$ for $\Sym^2 S$:
\def\boxpad{0.5}
\def\boxpadh{2.5*\boxpad}
\def\boxpadv{\boxpad}
\def\shift{0.25}
\def\midshift{\shift}
\def\topshift{2*\shift}
\def\hshift{0}
\[
\begin{tikzpicture}[xscale=0.65,yscale=0.65]
\node at (\hshift+0,0) {$\smallGLSrep{0}{0}$};
\node at (\hshift+1-\midshift,1) {$\smallGLSrep{1}{0}$};
\node at (\hshift+2-\topshift,2) {$\smallGLSrep{2}{0}$};
\node at (\hshift+3,0) {$\cdots$};
\node at (\hshift+4-\midshift,1) {$$};
\node at (\hshift+5-\topshift,2) {$\cdots$};
\node at (\hshift+6,0) {$\smallGLSrep{0}{6}$};
\node at (\hshift+7-\midshift,1) {$\smallGLSrep{1}{6}$};
\node at (\hshift+8-\topshift,2) {$\smallGLSrep{2}{6}$};

\draw[gray] (0-\boxpadh,0-\boxpadv) -- (2-\boxpadh,2+\boxpadv) -- (8+\boxpadh,2+\boxpadv) -- (6+\boxpadh,0-\boxpadv) -- cycle;
\end{tikzpicture}
\]

\begin{rem} The diagonal display format matches with~\cite{HoriKyoto} and is used later to visualize the direction of $\Hom$s in the collection: see Remark~\ref{rem.diags}.
\end{rem}

It will be helpful instead to take a dual exceptional collection as follows.\footnote{Note that this is an exceptional collection because, for locally free sheaves $\cE$ and $\cF$, we have $\Ext^\bullet(\cE,\cF)=\Ext^\bullet\big(\cF^\vee,\cE^\vee\big)$.}
\def\lowshift{1.5*\shift}
\def\midshift{0*\shift}
\def\topshift{-2*\shift}
\def\hshift{0.2}
\setcounter{col}{-1}
\begin{col}
\label{cola}
\begin{equation*}
\begin{tikzpicture}[xscale=0.65,yscale=0.65]

\node at (\hshift+-6-\lowshift,0) {$\smallGLSrepnew{0}{-6}$};
\node at (\hshift+-7-\midshift,1) {$\smallGLSrepnew{1}{-6}$};
\node at (\hshift+-8-\topshift,2) {$\smallGLSrepnew{2}{-6}$};
\node at (\hshift+-3-\lowshift,0) {$\cdots$};
\node at (\hshift+-4-\midshift,1) {$$};
\node at (\hshift+-5-\topshift,2) {$\cdots$};
\node at (\hshift+0-\lowshift,0) {$\smallGLSrepnew{0}{0}$};
\node at (\hshift+-1-\midshift,1) {$\smallGLSrepnew{1}{0}$};
\node at (\hshift+-2-\topshift,2) {$\smallGLSrepnew{2}{0}$};

\draw[gray] (-6-\boxpadh,0-\boxpadv) -- (-8-\boxpadh,2+\boxpadv) -- (-2+\boxpadh,2+\boxpadv) -- (\boxpadh,0-\boxpadv) -- cycle;

\end{tikzpicture}
\end{equation*}
\end{col}

\noindent From this collection, we deduce others by mutation. Namely, fixing an order for the collection $\langle \cE_1, \dots, \cE_n \rangle$, it is standard that mutating $\cE_n$ from right to left gives \begin{equation*}
L_1 \cdots L_{n-1} (\cE_n) = \Serre(\cE_n),
\end{equation*}
where the $L_i$ denote left mutation functors, and $\Serre$ is the Serre functor
\[
\Serre=-\otimes \omega_{\Gr}[\dim \Gr]=-\otimes\cO(-7)[\dim \Gr].
\]
Mutating in this way at $\cO$, and noting that the property of being an exceptional collection is invariant under the shift $[\dim \Gr]$, we obtain the following:
\def\lowshift{3.5*\shift}
\begin{col}
\label{colb}
\begin{equation*}
\begin{tikzpicture}[xscale=0.65,yscale=0.65]

\node at (\hshift+-7-\lowshift,0) {$\smallGLSrepnew{0}{-7}$};
\node at (\hshift+-7-\midshift,1) {$\smallGLSrepnew{1}{-6}$};
\node at (\hshift+-8-\topshift,2) {$\smallGLSrepnew{2}{-6}$};
\node at (\hshift+-4-\lowshift,0) {$\cdots$};
\node at (\hshift+-4-\midshift,1) {$$};
\node at (\hshift+-5-\topshift,2) {$\cdots$};
\node at (\hshift+-1-\lowshift,0) {$\smallGLSrepnew{0}{-1}$};
\node at (\hshift+-1-\midshift,1) {$\smallGLSrepnew{1}{0}$};
\node at (\hshift+-2-\topshift,2) {$\smallGLSrepnew{2}{0}$};

\draw[gray] (-8-\boxpadh,0-\boxpadv) -- (-7-\boxpadh,1) -- (-8-\boxpadh,2+\boxpadv) -- (-2+\boxpadh,2+\boxpadv) -- (-1+\boxpadh,1) -- (-2+\boxpadh,0-\boxpadv) -- cycle;

\end{tikzpicture}
\end{equation*}
\end{col}
\noindent Mutating at $S$ then gives:
\def\boxpadh{2.75*\boxpad}
\def\hshift{0.5}
\def\midshift{2*\shift}
\begin{col}
\label{colc}
\begin{equation*}
\begin{tikzpicture}[xscale=0.65,yscale=0.65]

\node at (\hshift+-7-\lowshift,0) {$\smallGLSrepnew{0}{-7}$};
\node at (\hshift+-8-\midshift,1) {$\smallGLSrepnew{1}{-7}$};
\node at (\hshift+-8-\topshift,2) {$\smallGLSrepnew{2}{-6}$};
\node at (\hshift+-4-\lowshift,0) {$\cdots$};
\node at (\hshift+-5-\midshift,1) {$$};
\node at (\hshift+-5-\topshift,2) {$\cdots$};
\node at (\hshift+-1-\lowshift,0) {$\smallGLSrepnew{0}{-1}$};
\node at (\hshift+-2-\midshift,1) {$\smallGLSrepnew{1}{-1}$};
\node at (\hshift+-2-\topshift,2) {$\smallGLSrepnew{2}{0}$};

\draw[gray] (-7-\boxpadh,0-\boxpadv) -- (-8-\boxpadh,1) -- (-7-\boxpadh,2+\boxpadv) -- (-1+\boxpadh,2+\boxpadv) -- (-2+\boxpadh,1) -- (-1+\boxpadh,0-\boxpadv) -- cycle;

\end{tikzpicture}
\end{equation*}
\end{col}

\noindent We now have two choices for how to mutate, at $\cO(-1)$ or $S^{2}$.


\noindent Mutating at $S^{2}$ gives a twist by $\cO(-1)$ of Collection~\ref{cola}:

\def\hshift{0.6}
\def\topshift{1*\shift}
\begin{col}
\label{cold}
\begin{equation*}
\begin{tikzpicture}[xscale=0.65,yscale=0.65]

\node at (\hshift+-6-\lowshift,0) {$\smallGLSrepnew{0}{-7}$};
\node at (\hshift+-7-\midshift,1) {$\smallGLSrepnew{1}{-7}$};
\node at (\hshift+-8-\topshift,2) {$\smallGLSrepnew{2}{-7}$};
\node at (\hshift+-3-\lowshift,0) {$\cdots$};
\node at (\hshift+-4-\midshift,1) {$$};
\node at (\hshift+-5-\topshift,2) {$\cdots$};
\node at (\hshift+0-\lowshift,0) {$\smallGLSrepnew{0}{-1}$};
\node at (\hshift+-1-\midshift,1) {$\smallGLSrepnew{1}{-1}$};
\node at (\hshift+-2-\topshift,2) {$\smallGLSrepnew{2}{-1}$};

\draw[gray] (-6-\boxpadh,0-\boxpadv) -- (-8-\boxpadh,2+\boxpadv) -- (-2+\boxpadh,2+\boxpadv) -- (\boxpadh,0-\boxpadv) -- cycle;

\end{tikzpicture}
\end{equation*}
\end{col}
\noindent Mutating at $\cO(-1)$ gives:
\def\hshift{0.25}
\def\shift{0.25}
\def\midshift{\shift}
\def\topshift{2*\shift}

\begin{col}
\label{cole}
\[
\begin{tikzpicture}[xscale=0.65,yscale=0.65]
\node at (\hshift+0,0) {$\smallGLSrepnew{0}{-8}$};
\node at (\hshift+1-\midshift,1) {$\smallGLSrepnew{1}{-7}$};
\node at (\hshift+2-\topshift,2) {$\smallGLSrepnew{2}{-6}$};
\node at (\hshift+3,0) {$\cdots$};
\node at (\hshift+4-\midshift,1) {$$};
\node at (\hshift+5-\topshift,2) {$\cdots$};
\node at (\hshift+6,0) {$\smallGLSrepnew{0}{-2}$};
\node at (\hshift+7-\midshift,1) {$\smallGLSrepnew{1}{-1}$};
\node at (\hshift+8-\topshift,2) {$\smallGLSrepnew{2}{0}$};

\draw[gray] (0-\boxpadh,0-\boxpadv) -- (2-\boxpadh,2+\boxpadv) -- (8+\boxpadh,2+\boxpadv) -- (6+\boxpadh,0-\boxpadv) -- cycle;
\end{tikzpicture}
\]
\end{col}
Noting then that being a full exceptional collection is invariant under twisting by $\cO(k)$, we obtain further collections summarized 
in the following.

\begin{prop}\label{prop.exc} Let $\underline{m}=(m_0,m_1,m_2)$ be a non-decreasing sequence of integers such that $m_{l+1} \leq m_l + 1$. Then
\begin{equation}
\WindowSetNew{\Sym^l S (m)}{2}{m_l-7,m_l}\end{equation}
gives a full exceptional collection on~$\Gr(2,V)$, for $\dim V = 7$.
\end{prop}

\section{Structure of equivalence proof}
\label{sec.struct}

I outline the proof of the derived equivalences $\Psi^{\underline{m}}$ between $Y_{\Gr}$ and $Y_{\Pf}$ in Theorem~\ref{mainthm.equiv}, before giving the proof in the following Sections~\ref{sec.equivLG} and~\ref{sec.relating}.


These derived equivalences are obtained by showing that $Y_{\Gr}$ and $Y_{\Pf}$ are derived equivalent, in an appropriate sense, to Landau--Ginzburg models
\begin{equation*}(X_{\Gr}, f) \qquad\text{and}\qquad (X_{\Pf}, f).\end{equation*}
Here the space $X_{\Gr}$ is a bundle over $\Gr(2,V)$, the space $X_{\Pf}$ is an Artin stack which is described in the next section, and $f$ is a function defined on both these spaces: $f$~arises by restriction of a~function from an Artin stack $\mfX$, of which $X_{\Gr}$ and $X_{\Pf}$ are open substacks. This follows a~physics construction of Hori--Tong~\cite{HT}, where this function is a ``superpotential''.


Each equivalence $\Psi^{\underline{m}}$ is a composition of three equivalences as below, where $\D(X_{\Gr},f)$ is a~category of matrix factorizations (see for instance \cite[Section~2]{ADS}), and $\Br^{\underline{m}}(X_{\Pf},f)$ is a closely related ``B-brane category'' (Definition~\ref{def.bbrane}):
\begin{equation*}
\label{threeequivs}
\begin{tikzpicture}[scale=2.5,xscale=1.5,yscale=0.7]
\node (A) at (0,0) {$\D(X_{\Gr}, f) $};
\node (B) at (1,0) {$\Br^{\underline{m}}(X_{\Pf},f) $};
\node (C) at (0,1) {$\D(Y_{\Gr})$};
\node (D) at (1,1) {$\D(Y_{\Pf})$};
\draw[->] (A) -- node[below] {$ \scriptstyle \Psi_{\cW}^{\underline{m}}$} node[above] {$ \scriptstyle \sim $} (B);
\draw[->] (C) -- node[left] {$ \scriptstyle \Psi_{\Gr} $} node[above,sloped] {$ \scriptstyle \sim $}(A);
\draw[->] (D) -- node[right] {$ \scriptstyle \Psi_{\Pf}^{\underline{m}} $} node[below,sloped] {$ \scriptstyle \sim $}(B);
\draw[dashed,->] (C) -- node[above] {$ \scriptstyle \Psi^{\underline{m}}$}(D);
\end{tikzpicture}
\end{equation*}

The vertical arrows are constructed as in~\cite{ADS}, using that the spaces~$Y$ may be recovered from the spaces $X$. In particular, $Y_{\Gr}$ is the critical locus of $f$ in~$X_{\Gr}$, giving the left-hand equivalence $\Psi_{\Gr}$ via Kn\"orrer periodicity~\cite{Shipman}. The right-hand equivalence $\Psi_{\Pf}^{\underline{m}}$ is a generalized Kn\"orrer periodicity, established in~\cite[Section~5]{ADS}.

The horizontal equivalence $\Psi_{\cW}^{\underline{m}}$ is a ``window equivalence'', which factors via a certain sub\-ca\-tegory in $\D(\mfX,f)$, determined by one of the exceptional collections from Proposition~\ref{prop.exc}. This is explained and proved in the following Section~\ref{sec.equivLG}. The rest of the proof is then given in~Section~\ref{sec.relating}.

\begin{rem}
There is a general theory of window equivalences~\cite{BFK,HL}, however it does not yet apply in this case. In particular, it gives exceptional collections related to that of Kapranov~\cite{Kapranov} rather than those in Proposition~\ref{prop.exc}: this seems to make it unsuitable for giving an equivalence with the Calabi--Yau~$Y_{\Pf}$. For more discussion, see~\cite[Remark~4.12]{ADS}.\end{rem}

\begin{rem} In the physics literature~\cite[Section~7]{EHKR} a grade restriction window for $\underline{m}=(7,7,7)$ is associated with the derived equivalence of Borisov and \Caldararu\ \cite{BorCal}. It would be interesting to compare the latter with the~$\Psi^{\underline{m}}$ obtained in this paper.
\end{rem}

\section{Equivalences for Landau--Ginzburg models}
\label{sec.equivLG}

In this section, I construct the window equivalences $\Psi_{\cW}^{\underline{m}}$ used to prove Theorem~\ref{mainthm.equiv}, after reviewing the construction of $X_{\Gr}$ and $X_{\Pf}$, the underlying spaces of the Landau--Ginzburg models discussed above.

Let $S$ be a $2$-dimensional vector space, and consider the quotient stack
\begin{equation*}\mfG = \quotstack{\Hom(S,V)}{\GL(S)}.\end{equation*}
The open substack of $\mfG$ of rank $2$ homomorphisms is equivalent to the variety $\Gr(2,V)$. Furthermore, the vector bundle on $\mfG$ induced by the representation~$S$ of $\GL(S)$ restricts to the tautological subspace bundle on $\Gr(2,V)$.
The representation $\det S^\vee$ similarly corresponds to the bundle $\cO(1)$ on $\Gr(2,V)$.

Take the quotient stack
\begin{equation*}\mfX = \quotstack{\Hom(S,V)\oplus \Hom\big( V, \wedge^2 S\big)}{\GL(S)}.\end{equation*}
Representations of $\GL(S)$ also give vector bundles on $\mfX$, and its substacks. In particular, the representation~$S$ gives a bundle on~$\mfX$, and the notation~$S$ will also be used to denote this bundle. For brevity, we write the line bundle corresponding to the representation $\big(\det S^\vee\big)^{\otimes k}$ as $\cO(k)$, by analogy with the notation on $\Gr$.

Now, with $x\in \Hom(S,V)$ and $p\in \Hom\big( V, \wedge^2 S\big)$, let
\begin{equation*}
X_{\Gr} = \{ \rk x = 2 \} \qquad\text{and}\qquad
X_{\Pf} = \{ \rk p = 1 \}
\end{equation*}
be open substacks of $\mfX$. These have the structure of a vector bundle over
$\Gr(2,V)$ and an Artin stack
$\mfP$, respectively, where
\begin{equation*}
\mfP = \big[\Hom^{\rk=1}\big(V, \wedge^2 S\big)\big/\GL(S)\big].
\end{equation*}

We let $f$ denote a function on the stack $\mfX$, and use the same notation for its restriction to~substacks. As the particular form of $f$ is not used in this section, we defer its definition to the following Section~\ref{sec.relating}.

The following categories correspond to the collections of Proposition~\ref{prop.exc}.
\begin{defn}[window subcategories $\cW^{\underline{m}}$]\label{def.window} Let $\underline{m}=(m_0,m_1,m_2)$ be a non-decreasing sequ\-ence of integers such that $m_{l+1} \leq m_l + 1$. Then
\begin{equation}\label{eq:rectangularcollection}\WindowSetNew{\Sym^l S (m)}{2}{m_l-7,m_l}\end{equation}
gives a set of bundles on $\mfX$, and we write $\cW^{\underline{m}}$ for the full subcategories of $\D(\mfX)$ and $\D(\mfX,f)$ generated by this set.
\end{defn}

For brevity, we notate bundles as follows.
\begin{notn}
Let $S_{l,m}$ denote the bundle $\Sym^l S(m)$ on $\mfX$, or its restriction to a substack.
\end{notn}

\begin{rem}\label{rem.adsgens} In \cite{ADS} a set of generators $T_{l,m}= \Sym^l S^\vee(m)$ was used. Note that $T_{l,m}\cong S_{l,m+l}$, using that $S^\vee\cong S(1)$ because $S^\vee \cong S \otimes \det S^\vee$. The window subcategory used in \cite{ADS}, with generators given in~\eqref{eq:rectangularcollectionnew} above, is therefore the window $\cW^{\underline{m}}$ with $\underline{m}=(6,7,8)$ in our notation here.\end{rem}

We write embeddings of stacks as follows:
\begin{equation*}X_{\Gr} \stackrel{i_{\Gr}}{\longrightarrow} \mfX \stackrel{i_{\Pf}}{\longleftarrow} X_{\Pf}.
\end{equation*}

\begin{prop}\label{prop.equivs}
For $ \cW^{\underline{m}} \subset \D(\mfX,f)$ the derived functor
\begin{equation*}
i_{\Gr}^* \colon\ \cW^{\underline{m}} \to \D(X_{\Gr}, f)
\end{equation*}
is an equivalence, and the derived functor
\begin{equation*}
i_{\Pf}^* \colon\ \cW^{\underline{m}} \to \D(X_{\Pf}, f)
\end{equation*}
is fully faithful.
\end{prop}

\begin{proof} We first prove the analogous statement with $\D(X_{\Gr}, f)$ and~$\D(X_{\Pf}, f)$ replaced with $\D(X_{\Gr})$ and~$\D(X_{\Pf})$, before explaining how to deduce the proposition.

To prove fully faithfulness, we generalize the argument given in~\cite[Lemma~4.3]{ADS}. It suffices to check fully faithfulness on generators: take then two of the generators $S_{l,m}$ and~$S_{l',m'}$ of~$\cW^{\underline{m}} $. The $\Hom$s between these are equal to the $\Hom$s between their restrictions to $X_{\Gr}$ and $X_{\Pf}$, respectively, because the complements of the latter have codimensions at least~$2$. Furthermore, there are no higher $\Ext$s between them because they are vector bundles on an affine stack, so it suffices to check that they do not acquire any higher $\Ext$s after applying the functors $i^*$. Namely, we require the following:
\begin{gather}
\Ext^{>0}_{X_{\Gr}} \big(i_{\Gr}^*S_{l,m},\; i_{\Gr}^*S_{l', m'}\big) = 0, \label{eqn.extG}
\\
\Ext^{>0}_{X_{\Pf}} \big(i_{\Pf}^*S_{l,m},\; i_{\Pf}^*S_{l', m'}\big) = 0. \label{eqn.extP}
\end{gather}
These vanishings are proved in Sections~\ref{sec.vanG} and~\ref{sec.vanP}, respectively.

We then claim that $i_{\Gr}^*$ is essentially surjective. For this, note that the generators $S_{l,m}$, when considered as sheaves on $\Gr(2,V)$ give generators for the derived category by Proposition~\ref{prop.exc}. Essential surjectivity then follows by a general argument, exactly as in~\cite[Lemma~4.6]{ADS}.

To complete the proof of the proposition, we repeat the arguments in \cite[Proposition~4.9, Lemma~4.10]{ADS}. The important point here is that morphisms in the categories $\D(X_{\Gr}, f)$ and $\D(X_{\Pf}, f)$ are related to those in the categories $\D(X_{\Gr})$ and~$\D(X_{\Pf})$ by certain spectral sequ\-ences.
\end{proof}

Using the above proposition we have the following.

\begin{defn}[B-brane categories]\label{def.bbrane}
We define a subcategory
\begin{equation*}\Br^{\underline{m}}(X_{\Pf}, f) \subset \D(X_{\Pf}, f) \end{equation*}
to be the image of $\cW^{\underline{m}} \subset \D(\mfX, f)$ under $i_{\Pf}^*$.
\end{defn}

\begin{rem} In other words, $\Br^{\underline{m}}(X_{\Pf}, f)$ is the full subcategory of $\D(X_{\Pf}, f) $ generated by the bundles given by~\eqref{eq:rectangularcollection}.\end{rem}

\begin{defn}[window equivalences]\label{def.windowequiv}
We write
\begin{equation*}
\Psi_{\cW}^{\underline{m}} = i_{\Pf}^* \circ (i_{\Gr}^*)^{-1} \colon\ \D(X_{\Gr},f) \to \cW^{\underline{m}} \to \Br^{\underline{m}}(X_{\Pf},f).
\end{equation*}
\end{defn}

\subsection{Grassmannian side}\label{sec.vanG} I start with some vanishing results on $\Gr(2,V)$, before proving the vanishing~\eqref{eqn.extG} on $X_{\Gr}$. I\opt{s}{~}\opt{a}{ }recall the following, again writing $T_{l,m}= \Sym^l S^\vee(m)$.

\begin{lem}[{\cite[Lemma 4.5]{ADS}}]\label{lem.Kuznetsov_extended_new}
Let $\Gr=\Gr(2,V)$, with $\dim V=n$ odd. If $0 \leq l, l' \leq \frac{1}{2}n-1$ and~$m' \geq m$ then we have
\begin{equation*}
\Ext^{>0}_{\Gr} \big(T_{l,m},\, T_{l',m'} \big) = 0.
\end{equation*}
\end{lem}

We quickly obtain the following, where again $S_{l,m}= \Sym^l S(m)$.

\begin{cor} \label{cor.Kuznetsov_extended_new_mod}
In the setting of Lemma~$\ref{lem.Kuznetsov_extended_new}$ above, so that in particular we have~$m' \geq m$,
\begin{equation*}
\Ext^{>0}_{\Gr} \big(S_{l,m},\, S_{l',m'} \big) = 0\qquad\text{and}\qquad \Ext^{>0}_{\Gr} \big(S_{l,m},\, S_{l',m'+l'-l} \big)=0.
\end{equation*}
\end{cor}
\begin{proof} For the first, we exchange $l$ and $l'$, and note that the result is isomorphic to the $\Ext$-group in Lemma~\ref{lem.Kuznetsov_extended_new} above, using local freeness of the $\Sym$s. For the second, we use that $T_{l,m}\cong S_{l,m+l}$, and the local freeness of $\cO(1)$.
\end{proof}
\begin{rem}\label{rem.diags} In terms of the following picture, this corollary says that there are no higher $\Ext$s on $\Gr(2,V)$ from an object on the lines pictured, say~$L$, to any object on a parallel line $L'$ lying to the right of~$L$:
\def\boxpad{0.5}
\def\boxpadh{2.5*\boxpad}
\def\boxpadv{\boxpad}
\def\shift{0.1}
\def\midshift{\shift}
\def\lowshift{6*\shift}
\def\topshift{-6*\shift}
\[
\begin{tikzpicture}[xscale=1.3,yscale=0.65] 

\draw[gray] (2-\lowshift,0-\boxpadv) -- (0-\topshift,2+\boxpadv);
\draw[gray] (1-\lowshift,0-\boxpadv) -- (-1-\topshift,2+\boxpadv);
\draw[gray] (0-\lowshift,0-\boxpadv) -- (-2-\topshift,2+\boxpadv);

\draw[gray] (1-\lowshift,0-\boxpadv) -- (1-\topshift,2+\boxpadv);
\draw[gray] (0-\lowshift,0-\boxpadv) -- (0-\topshift,2+\boxpadv);
\draw[gray] (-1-\lowshift,0-\boxpadv) -- (-1-\topshift,2+\boxpadv);

\node at (-2-\lowshift,0) {$\cdots$};
\node at (-3-\topshift,2) {$\cdots$};

\node at (-1-\lowshift,0) {$\smallGLSrepnewabbr{0}{-1}$};

\node at (0-\lowshift,0) {$\smallGLSrepnewabbr{0}{0}$};
\node at (-1-\midshift,1) {$\smallGLSrepnewabbr{1}{0}$};
\node at (-2-\topshift,2) {$\smallGLSrepnewabbr{2}{0}$};

\node at (1-\lowshift,0) {$\smallGLSrepnewabbr{0}{1}$};
\node at (0-\midshift,1) {$\smallGLSrepnewabbr{1}{1}$};
\node at (-1-\topshift,2) {$\smallGLSrepnewabbr{2}{1}$};

\node at (2-\lowshift,0) {$\smallGLSrepnewabbr{0}{2}$};
\node at (1-\midshift,1) {$\smallGLSrepnewabbr{1}{2}$};
\node at (0-\topshift,2) {$\smallGLSrepnewabbr{2}{2}$};

\node at (1-\topshift,2) {$\smallGLSrepnewabbr{2}{3}$};

\node at (3-\lowshift,0) {$\cdots$};
\node at (2-\topshift,2) {$\cdots$};

\end{tikzpicture}
\]
\end{rem}

I now show the vanishing \eqref{eqn.extG}. By a standard calculation using the structure of $X_{\Gr}$ as a~vector bundle over $\Gr$, we have
\begin{gather*}
\RHom_{X_{\Gr}} \big(i_{\Gr}^*S_{l,m},\, i_{\Gr}^*S_{l',m'} \big) \cong \bigoplus_{n\geq 0} \RHom_{\Gr} \big( S_{l,m}, \, S_{l',m'} \otimes \Sym^n \cO(1)^{\oplus 7} \big).
\end{gather*}
By expanding the $\Sym^n$ piece into irreducibles we see that this splits into summands
\begin{equation*}
\RHom_{\Gr} \big(S_{l,m},\, S_{l',m'+n} \big),
\end{equation*}
where $S_{l,m}, S_{l',m'} \in \cW^{\underline{k}} $ and $n\geq 0$. It therefore suffices to know that there are no higher $\Ext$s from the generators $S_{l,m}$ of~$\cW^{\underline{m}} $ to other generators $S_{l',m'}$ of~$\cW^{\underline{m}}$ or to sheaves $S_{l',m'+n}$ ``to their right''. But this is clear from the form of the exceptional collections in Section~\ref{sec.cols}, combined with Corollary~\ref{cor.Kuznetsov_extended_new_mod} above.

\subsection{Pfaffian side}\label{sec.vanP}
We now show the vanishing \eqref{eqn.extP}. Similarly to the above, by a standard calculation we have
\begin{gather*}
\RHom_{X_{\Pf}} \big(i_{\Pf}^*S_{l,m},\,i_{\Pf}^*S_{l',m'}\big)\cong \bigoplus_{n\geq 0} \RHom_{\mathcal{P}}\big(S_{l,m}, \,S_{l', m'} \otimes \Sym^n S^{\oplus 7} \big).
\end{gather*}
We check the following very straightforward bound. This will give the vanishing for the $\Sym^0$ piece: the vanishing for the $\Sym^n $ piece with $n>0$ will then follow by induction.

\begin{lem}\label{lem.induct}
For two generators
$S_{l,m}$ and $S_{l', m'} $ of $\cW^{\underline{m}} $ we have
\begin{equation*}
m' - m < \max(l'-l,0)+7.
\end{equation*}
\end{lem}

\begin{proof} By assumption
\[
m\in (m_l-7,m_l]\qquad\text{and}\qquad
m'\in (m_{l'}-7,m_{l'}].
\]
Notice then that
\begin{gather*}
m' - m < m_{l'} - (m_l - 7) = (m_{l'} - m_l) + 7.
\end{gather*}
Recall that the $m_l$ form a non-decreasing sequence with $m_{l+1}\leq m_l + 1$. If $l' \geq l$ then \begin{equation*}
m_{l'} - m_l \leq l' - l,
\end{equation*}
and otherwise $m_{l'} - m_l \leq 0$, hence the claim.
\end{proof}

\begin{rem}
The collections $\cW^{(0,0,0)}$ and $\cW^{(-2,-1,0)}$ (given as Collections~\ref{cola} and~\ref{cole} in~Section~\ref{sec.cols}) show that, amongst such bounds which are uniform for all the $\cW^{\underline{m}}$, the bound of Lemma~\ref{lem.induct} is the best possible.
\end{rem}

For our induction, we now observe that
\begin{align}\label{eq.induct_identity}
S_{l,m} \otimes S & = \Sym^l S(m) \otimes S \nonumber \\
& \cong \big( \Sym^{l+1} S \,\oplus\, \Sym^l S \otimes \wedge^2 S \,\big)(m) \nonumber \\
& \cong \Sym^{l+1} S(m) \,\oplus\, \Sym^l S(m-1) \nonumber \\
& = S_{l+1,m} \oplus S_{l,m-1}.
\end{align}
Now suppose we use this repeatedly to expand the $\Sym^n$ piece, namely
\begin{equation*}
\RHom_{\mathcal{P}}\big(S_{l,m},S_{l', m'} \otimes \Sym^n S^{\oplus 7} \big).
\end{equation*}
I claim we get summands
\begin{equation}\label{eqn.summand}
\RHom_{\mathcal{P}}\big(S_{l,m},S_{l'', m''} \big)
\end{equation} for which the inequality
\begin{equation}\label{eq.induct_summand}
m'' - m < \max(l''-l,0)+7.
\end{equation}
is satisfied. This follows by induction: the base case $n=0$ is shown in Lemma~\ref{lem.induct}, and the inequality still holds when we increment~$l''$ or decrement~$m''$ each time we apply \eqref{eq.induct_identity}.


Let us now calculate the cohomology of the summands
\eqref{eqn.summand}. First note
\begin{align*}\label{eq.expansion}
\RsHom_{\mathcal{P}}\big(S_{l,m},S_{l'', m''}\big) &\cong \Sym^l S^\vee(-m) \otimes \Sym^{l''} S (m'') \\
& \cong \Sym^l S(l-m) \otimes \Sym^{l''} S (m'')
\\
& \cong \big(\Sym^l S \otimes \Sym^{l''} S \big)(l-m+m'').
\end{align*}
Applying $\RDerived\Gamma_{\mathcal{P}}$ to this, we will obtain $\RHom_{\mathcal{P}} \cong \RDerived\Gamma_{\mathcal{P}} \circ \RsHom_{\mathcal{P}}$. Observe then that there is a~morphism of stacks $\delta\colon \mathcal{P}\to \P^6$ induced by $\det\colon\GL(S) \to \C^*$, and we have
$\RDerived\Gamma_{\mathcal{P}} \cong \RDerived\Gamma_{\P^6} \circ \delta_*.$
The functor $\delta_*$ takes $\ker(\det)$-invariants where $\ker(\det) = \SL(S)$. Expanding $\Sym^l S \otimes \Sym^{l''}S $ into irreducibles, the only sheaf that survives this is $(\wedge^2 S)^{\otimes l}\cong\cO(-l)$ in the case that $l=l''$.

We thence deduce that the only contribution to the cohomology of \eqref{eqn.summand} comes from
\opt{a}{$\RDerived\Gamma_{\mathcal{P}} \cO(-m+m'')$}
\opt{s}{\begin{equation*}\RDerived\Gamma_{\mathcal{P}} \cO(-m+m'')\end{equation*}}
in the case $l=l''$ and that, setting $n=m'' - m$, this contribution is
\begin{equation}\label{eq.pushdown}\RDerived\Gamma_{\P^6} \delta_* \cO(n)
\cong \RDerived\Gamma_{\P^6} \cO_{\P^6}(-n).
\end{equation}
But \eqref{eq.induct_summand} implies that if $l=l''$ then $n< 7$, so the required vanishing of higher $\Ext$s in \eqref{eqn.extP} follows from standard cohomology vanishing on $\P^6$.

\begin{rem}\label{rem.sign}
The minus sign in~\eqref{eq.pushdown} appears because the bundle $\cO(1)$ on
\begin{equation*}
\mathcal{P} = \quotstack{\Hom\big( V, \wedge^2 S\big) - \{0\}}{\GL(S)}
\end{equation*}
corresponds to the representation $\wedge^2 S^\vee$ by the definitions at the beginning of this section. The bundle $\cO_{\P^6}(1)$, on the other hand, corresponds to $\wedge^2 S$.
\end{rem}

\begin{rem}
The last part of the argument above, the calculation of the cohomology of the summands~\eqref{eqn.summand}, parallels the proof of~\cite[Lemma~4.3]{ADS}. We include it here to explain some details, and for convenience.\end{rem}

\section[Equivalences for Calabi--Yau 3-folds]{Equivalences for Calabi--Yau $\boldsymbol 3$-folds}\label{sec.relating}

We now show that the categories $\D(X_{\Gr},f)$ and $\Br^{\underline{m}}(X_{\Pf},f)$ are equivalent, respectively, to the derived categories of our Calabi--Yau $3$-folds $Y_{\Gr}$ and~$Y_{\Pf}$, completing the proof of Theorem~\ref{mainthm.equiv}.

In this section we require the definition of the function $f$ on $\mfX$ from~\cite{ADS}, as follows. Take a~surjective map $A\colon \wedge^2 V \to V$ such that the dual subspaces $\Pi$ and $\Pi^\circ$ used to construct the $3$-folds in Section~\ref{sec.threefolds} are given by
\[\Pi = \ker A^\vee \qquad\text{and}\qquad \Pi^\circ = \Im A.\]
Recalling that
\begin{equation*}
\mfX = \quotstack{\Hom(S,V)\oplus \Hom\big(V, \wedge^2 S\big)}{\GL(S)}
\end{equation*}
with $x\in \Hom(S,V)$ and $p\in \Hom\big( V, \wedge^2 S\big)$ take a function on $\mfX$ defined by
\[
f(x,p) = p \circ A \circ \wedge^2 x
\]
via the canonical isomorphism $\Hom\big({\wedge}^2 S, \wedge^2 S\big)\cong \mathbb{C}.$

In addition to the data $(\mfX,f)$, the Landau--Ginzburg model includes the data of a $\mathbb{C}^*$-action on $\mfX$ for which $f$ has weight~$2$. For this, we let $\mathbb{C}^*$ act with weight~$0$ on~$x$, and weight~$2$ on~$p$. Note that this action has been suppressed in our notation so far. Following physics terminology, it is sometimes know as the ``R-charge''. For more details, see~\cite[Section~2.1]{ADS}.

\subsection{Grassmannian side}

Recall that $X_{\Gr}$ is by definition the Artin stack
\[ X_{\Gr} = \quotstackBig{ \setcondsbig{(x,p) \in \Hom(S,V) \oplus \Hom\big(V,\wedge^2 S\big)}{\rk(x)=2}}{\GL(S)}, \]
which is isomorphic to the total space of the vector bundle $\cO(-1)^{\oplus 7}$ over the Grassmannian~$\Gr(2, V)$. Denote the projection morphism for this bundle by~$\pi$. The Calabi--Yau 3-fold $Y_{\Gr}$ defined in Section~\ref{sec.threefolds} is the zero locus of a transverse section of $\cO(1)^{\oplus 7}$ given by
\begin{equation*} s = A\circ \wedge^2 x.\end{equation*}
Then we have $f=sp$, where
$p$ is the tautological section of $\pi^*\cO(-1)^{\oplus 7}$. The $\mathbb{C}^*$-action preserves each fibre of the bundle, and acts with weight~$2$ on those fibres.

Given this geometric situation, an equivalence from $\D(Y_{\Gr})$ to $\D(X_{\Gr},f)$ follows from a~version of Kn\"orrer periodicity, as follows. Let $X_{\Gr}|_{Y_{\Gr}}$ denote the base change of the bundle~$X_G$ over~$G$ to the base~$Y_G$. Then we have
\begin{equation*}
Y_{\Gr} \stackrel{\pi}{\longleftarrow} X_{\Gr}|_{Y_{\Gr}} \stackrel{k}{\longrightarrow} X_{\Gr},
\end{equation*}
where $k$ is the inclusion.

\begin{defn}\label{def.equivGr}
We define an equivalence $\Psi_{\Gr}$ by the composition
\begin{equation*}
\Psi_{\Gr} \colon \ \D(Y_{\Gr}) \stackrel{\pi^*}{\longrightarrow} \D(X_{\Gr}|_{Y_{\Gr}}) \stackrel{k_*}{\longrightarrow} \D(X_{\Gr}, f).
\end{equation*}
\end{defn}

This is an equivalence by a result of Shipman~\cite[Theorem~3.4]{Shipman}.

\subsection{Calculations} Later, in Section~\ref{sec twists}, I study certain bundles on~$Y_{\Gr}$ to prove Theorem~\ref{keythm.twists}: for use there, in this subsection I calculate the images of these objects under $\Psi_{\Gr}$. I first give the method in a simpler case, before applying it to $Y_{\Gr}$.

\begin{eg} Consider a line bundle $L$ over a variety $B$, and a subvariety $Y \subset B$ cut out by a transverse section $s$ of $L$. Let $X$ be the total space of~$L^\vee$, with projection $\pi$, and $p$ the tautological section of $\pi^* L^\vee$. We make the setup $\C^*$-equivariant so that $s$ and $p$ have weights~$0$ and~$2$, respectively. Take notation as follows, where $X|_Y$ is the total space of the base change of~the bundle $L$ over $B$ to $Y$, and $j$ is the inclusion of the zero section $B$ in $X$:
\begin{equation*}
\begin{tikzpicture}[scale=1.5,xscale=1.1]
\node (A) at (0,0) {$X$};
\node (B) at (-1,0) {$X|_Y$};
\node (C) at (0,-1) {$B$};
\node (D) at (-1,-1) {$Y$};
\draw[right hook->] (B) -- node[above] {$ \scriptstyle k$} (A);
\draw[right hook->] (D) -- node[above] {$ \scriptstyle k $} (C);
\draw[transform canvas={xshift=-.4ex},<-] (C) -- node[left] {$\scriptstyle \pi$} (A);
\draw[transform canvas={xshift=.4ex},<-left hook] (A) -- node[right] {$\scriptstyle j $} (C);
\draw[<-] (D) -- node[left] {$\scriptstyle \pi$} (B);
\end{tikzpicture}
\end{equation*}
Then in $\D(X,f)$, where $f=sp$, we have an object $\mathcal{K}$ given by
\begin{equation*}
\begin{tikzpicture}[scale=2]
\node (A) at (0,0) {$\pi^*L^\vee[-1]$};
\node (B) at (1,0) {$\cO_X$,};
\draw[transform canvas={yshift=.4ex},->] (A) -- node[above] {$ \scriptstyle s$} (B);
\draw[transform canvas={yshift=-.4ex},->] (B) -- node[below] {$ \scriptstyle p$} (A);
\end{tikzpicture}
\end{equation*}
where the shift $[-1]$ denotes a change in $\C^*$-weight.
Forgetting the right- and left-moving morphisms respectively in $\mathcal{K}$ gives Koszul resolutions for $\C^*$-equivariant sheaves
\begin{equation}\label{eqn.koszuls}
\pi^* k_* \cO_Y \qquad\text{and}\qquad j_* \cO_B \otimes \pi^* L^\vee [-1].
\end{equation}
Note that these determine objects of $\D(X,f)$ because they are supported on the zero locus of $f$, namely $X|_Y \cup B$. Indeed, they may be written as pushforwards from this locus using isomorphisms
\begin{equation*}
\pi^* k_* \cO_Y \cong k_* \pi^* \cO_Y \qquad\text{and}\qquad
 j_* \cO_B \otimes \pi^* L^\vee \cong j_* \big(\cO_B \otimes j^*\pi^* L^\vee\big) \cong j_* \big(\cO_B \otimes L^\vee\big)\end{equation*}
which follow by flat base change, and the projection formula for $j$, respectively. Now, by a~standard argument with matrix factorizations, both the objects of \eqref{eqn.koszuls} are isomorphic to $\mathcal{K}$ in~$\D(X,f)$: see for instance \cite[proof of Lemma~3.2]{Shipman}. Putting all this together, we obtain an~isomorphism in $\D(X,f)$ as follows.
\begin{equation*}
k_* \pi^* \cO_Y \cong j_* (\cO_B \otimes L^\vee)[-1].
\end{equation*}\end{eg}

We generalize the method of this example to deduce the following.
\begin{prop}\label{prop.images} We have that
\begin{align*}
\Psi_{\Gr}\colon\ \cO_{Y_{\Gr}} & \mapsto j_* \cO_{\Gr}(-7)[-7] \cong j_! \cO_{\Gr},
\end{align*}
where $j$ is the inclusion of the zero section of $X_{\Gr}$, and similarly with occurences of~$\cO$ replaced with $S$ or \,$\Sym^2 S$.
\end{prop}
\begin{proof}
Take $F =\cO(1)^{\oplus 7}$, a bundle on $\Gr(2,V)$ of rank $r=\dim V=7$, and set notation for morphisms as follows:
\begin{equation*}
\begin{tikzpicture}[scale=1.5,xscale=1.3]
\node (A) at (0,0) {$X_{\Gr}$};
\node (B) at (-1,0) {$X_{\Gr}|_{Y_{\Gr}}$};
\node (C) at (0,-1) {$\Gr$};
\node (D) at (-1,-1) {$Y_{\Gr}$};
\draw[right hook->] (B) -- node[above] {$ \scriptstyle k$} (A);
\draw[right hook->] (D) -- node[above] {} (C);
\draw[transform canvas={xshift=-.4ex},<-] (C) -- node[left] {$\scriptstyle \pi$} (A);
\draw[transform canvas={xshift=.4ex},<-left hook] (A) -- node[right] {$\scriptstyle j $} (C);
\draw[<-] (D) -- node[left] {$\scriptstyle \pi$} (B);
\end{tikzpicture}
\end{equation*}
Following the argument of the example above, we take $\mathcal{K}$ in~$\D(X_{\Gr},f)$ as follows:
\begin{equation*}
\begin{tikzpicture}[scale=2.1]
\node (P) at (-2.1,0) {$\wedge^r\pi^*F^\vee[-r]$};
\node (Q) at (-1,0) {$\dots$};
\node (A) at (0,0) {$\pi^*F^\vee[-1]$};
\node (B) at (1.1,0) {$\cO_{X_{\Gr}}$,};
\draw[transform canvas={yshift=.4ex},->] (A) -- node[above] {$ \scriptstyle s$} (B);
\draw[transform canvas={yshift=-.4ex},->] (B) -- node[below] {$ \scriptstyle p$} (A);

\draw[transform canvas={yshift=.4ex},->] (P) -- node[above] {$ \scriptstyle s$} (Q);
\draw[transform canvas={yshift=-.4ex},->] (Q) -- node[below] {$ \scriptstyle p$} (P);

\draw[transform canvas={yshift=.4ex},->] (Q) -- node[above] {$ \scriptstyle s$} (A);
\draw[transform canvas={yshift=-.4ex},->] (A) -- node[below] {$ \scriptstyle p$} (Q);

\end{tikzpicture}
\end{equation*}
which gives an isomorphism of objects
\begin{equation*} k_* \pi^* \cO_{Y_{\Gr}} \cong j_* \big(\cO_{\Gr} \otimes \det F^\vee\big) [-r]. \end{equation*}
Recalling that $j_! = j_*(\omega_j[\dim j]\otimes - )$, and noting that $\omega_j = \det F^\vee$ and $\dim j=-r$, the first claim follows. Twisting $\mathcal{K}$ by $S$ or $\Sym^2 S$ then gives the further claim.\end{proof}

\subsection{Pfaffian side}

Recall that $X_{\Pf}$ is the Artin stack
\[
X_{\Pf} = \quotstackBig{ \setcondsbig{(x,p) \in \Hom(S,V) \oplus \Hom\big(V,\wedge^2 S\big)}{p \ne 0}}{\GL(S)}.
\]
Given $A\colon \wedge^2 V \to V$, each value of $p$ determines a $\wedge^2 S$-valued 2-form on $V$, namely $\twoform{p} = p \circ A$. We then have that
$f(x,p) = \omega_p \circ \wedge^2 x$ by the definition of~$f$.

In this subsection, we construct an equivalence
\[
\D(Y_{\Pf}) \to \Br^{\underline{m}}(X_{\Pf}, f)
\]
for each choice of $\underline{m}$. In \cite{ADS}, such an equivalence was constructed for $\underline{m}={(6,7,8)}$ by considering isotropic subspaces for the $\omega_p$ as $p$ varies. An outline is given in \cite[Section~5.1]{ADS} with details in~the subsections following. The argument is lengthy and subtle, but to obtain an equivalence for general~$\underline{m}$ I only need to modify the very last step of it: this subsection explains the modification.


The argument uses the following locus in $X_{\Pf}$. Note that $X_{\Pf}$ is a bundle over $\mathcal{P}$ and also over~$\P^6$ by composition with the morphism $\delta\colon \mathcal{P}\to \P^6$ induced by $\det\colon\GL(S) \to \C^*$.

\begin{defn}[{\cite[Definition~5.5]{ADS}}]\label{defn:Gamma}
Let $\Gamma \subset X_{\Pf}$ be the closed substack with points $(x,p)$, where~$p$ corresponds to a point of $Y_{\Pf} \subset \P^6$, and the composition of $x$ with a quotient morphism
\begin{equation*}
S \overset{x}{\longrightarrow} V \to V/\ker \twoform{p}
\end{equation*}
has rank at most 1.
\end{defn}

\begin{rem}Note that, when $\rk(x)=2$, the above condition matches the condition defining the correspondence between $Y_{\Gr}$ and $Y_{\Pf}$ which appears in Borisov and \Caldararu's work~\cite[Section~0.7]{BorCal}. They show that the ideal sheaf of this correspondence gives a derived equivalence.
\end{rem}

Note that $\Gamma$ is a flat family of stacks over $Y_{\Pf}$. The following proposition gives the construction of $\Psi_{\Pf}^{\underline{m}}$. Take notation as below, where $j$ and $k$ denote inclusions:
\begin{equation*}
\begin{tikzpicture}[scale=1.5,xscale=1.2]
\node (A) at (0,0) {$X_{\Pf}$};
\node (B) at (-1,0) {$X_{\Pf}|_{Y_{\Pf}}$};
\node (C) at (0,-1) {$\P^6$};
\node (D) at (-1,-1) {$Y_{\Pf}$};
\node (E) at (-2,0) {$\Gamma$};
\draw[right hook->] (B) -- node[above] {$ \scriptstyle k$} (A);
\draw[right hook->] (D) -- node[above] {$ \scriptstyle k$} (C);
\draw[<-] (C) -- node[left] {$\scriptstyle \pi$} (A);
\draw[<-] (D) -- node[left] {$\scriptstyle \pi$} (B);
\draw[right hook->] (E) -- node[above] {$\scriptstyle j$} (B);
\end{tikzpicture}
\end{equation*}

\begin{prop}\label{prop.Flandsinwindow}
There exists an equivalence
\begin{equation*}
\Psi_{\Pf}^{\underline{m}} = k_* \circ j_* \circ j^* \circ \pi^* \colon\ \D(Y_{\Pf}) \overset{\sim}{\longrightarrow} \Br^{\underline{m}}(X_{\Pf}, f) \subset \D(X_{\Pf}, f).
\end{equation*}
\end{prop}
\begin{proof}
I explain how choices can be made so that the given composition has essential image in~$\Br^{\underline{m}}(X_{\Pf}, f)$. This refines the argument of \cite[proof of Proposition 5.9]{ADS}.

I first explain why, for $\cE$ a sheaf on $Y_P$, we may make choices so that $\Psi_{\Pf}^{\underline{m}}(\cE) \in \Br^{\underline{m}}(X_{\Pf}, f)$. The result then follows by generation. Using that $\pi j$ is flat, $\Psi_{\Pf}^{\underline{m}}(\cE)$ is a sheaf on $X_{\Pf}$. Furthermore, by the projection formula,
\begin{equation*}
\Psi_{\Pf}^{\underline{m}}(\cE) = k_*j_*j^*{\pi}^*(\cE) \cong k_*(j_*\cO_\Gamma\otimes \pi^*\cE).
\end{equation*}
Now $j_*\cO_\Gamma$ has an Eagon--Northcott resolution
\begin{equation*}
\wedge^4Q^\vee\otimes \Sym^2 S(1) \to \wedge^3Q^\vee\otimes S(1)
\to \wedge^2Q^\vee(1) \to \cO,
\end{equation*}
where $Q=V/\ker \twoform{p}$, by for instance \cite[Section~6.1.6]{Weyman}. This may be made $\C^*$-equivariant by adding appropriate shifts of $\C^*$-weight, and it follows that $\Psi_{\Pf}^{\underline{m}}(\cE)$ has a $\C^*$-equivariant resolution
\begin{equation*}
\pi^*\cF_2\otimes \Sym^2 S \to \pi^*\cF_1\otimes S \to \pi^*\cF_0 \to \pi^*\cF ,
\end{equation*}
where $\cF$ and $\cF_0, \ldots, \cF_2$ are sheaves in the image of $j_*\colon \operatorname{Coh}(Y_{\Pf}) \to \operatorname{Coh}(\P^6)$.
Using a Beilinson collection on $\P^6$, each of these $\cF$s can be replaced by a quasi-isomorphic complex whose terms are direct sums of line bundles $\cO_{\P^6}(-m)$ with $m\in(n-7,n]$, for any choice of integer $n$.

Now $\pi^* \cO_{\P^6}(-m) \cong \cO(m)$, with the sign appearing as in Remark~\ref{rem.sign}. Therefore, we may obtain a resolution of $\Psi_{\Pf}^{\underline{m}}(\cE)$ by vector bundles in the window subcategory~$\cW^{\underline{m}}$, by letting $n=m_l$ for~$\cF_l$, and $n=m_0$ for~$\cF$. Doing likewise for $\cE$ ranging over a set of generators for~$\D(Y_{\Pf})$, it~follows that $\Psi_{\Pf}^{\underline{m}}$ may be constructed to have essential image in~$\Br^{\underline{m}}(X_{\Pf}, f)$. The proof of the equivalence property is exactly as in~\cite[Theorem~5.12]{ADS}.
\end{proof}

\begin{rem} A different construction of such a functor is given in \cite[Section~4.2]{SegalThomas}, avoiding much of the difficult analysis in~\cite{ADS}. However, the discussion in the latter suffices for our purposes.
\end{rem}

\subsection{Equivalences}

We now complete the proof of Theorem~\ref{mainthm.equiv}.


As before, let $\underline{m}=(m_0,m_1,m_2)$ be a non-decreasing sequence of integers such that $m_{l+1} \leq m_l + 1$.

\begin{thm}\label{thm.equiv}For each $\underline{m}$ as above, there is an equivalence
\begin{equation*}
\Psi^{\underline{m}}\colon\ \D(Y_{\Gr}) \overset{\sim}{\longrightarrow} \D(Y_{\Pf}),
\end{equation*}
which factors through the subcategory $\cW^{\underline{m}}$ of $\D(\mfX,f)$.
\end{thm}
\begin{proof}

Recall that in the previous section we obtained (Definition~\ref{def.windowequiv}) an equivalence as follows.
\begin{equation*}
\Psi_{\cW}^{\underline{m}} = i_{\Pf}^* \circ (i_{\Gr}^*)^{-1} \colon\ \D(X_{\Gr},f) \to \cW^{\underline{m}} \to \Br^{\underline{m}}(X_{\Pf},f).
\end{equation*}
Composing this with the equivalences of this section, from Definition~\ref{def.equivGr} and Proposition~\ref{prop.Flandsinwindow}, we get
\begin{equation*}
\Psi^{\underline{m}} = (\Psi_{\Pf}^{\underline{m}})^{-1} \circ \Psi_{\cW}^{\underline{m}} \circ \Psi_{\Gr} \colon\ \D(Y_{\Gr}) \to \D(Y_{\Pf})
\end{equation*}
and the theorem is proved.
\end{proof}

\section{Groupoid action}
\label{sec groupoid}

In this section I construct the groupoid action of Theorem~\ref{keythm.groupoid}. Let us first recall the notation which was used to state this theorem in Section~\ref{sec.introactions}.

\begin{notn}\label{notn.windows} For brevity, we write windows, and equivalences,
\begin{equation*}
\begin{split}
\cW^0 &= \cW^{(0,0,0)}, \\
\cW^1 &= \cW^{(-1,0,0)}, \\
\cW^2 &= \cW^{(-1,-1,0)}, \\
\cW^3 &= \cW^{(-1,-1,-1)},
\end{split}
\qquad\qquad
\begin{split}
\Psi^0 & = \Psi^{(0,0,0)}, \\
\Psi^1 & = \Psi^{(-1,0,0)}, \\
\Psi^2 & = \Psi^{(-1,-1,0)}, \\
\Psi^3 & = \Psi^{(-1,-1,-1)},
\end{split}
\end{equation*}
and similarly write $\Psi^l_{\cW}$ for $l=0,\dots,3.$
\end{notn}

Theorem~\ref{keythm.groupoid} will follow from the fact that $\cW^3 = \cW^0 \otimes \cO(-1)$, and the following proposition.

\begin{prop}\label{prop.cyinter}
There are natural isomorphisms as follows
\begin{gather*}
(-\otimes\cO_{X_{\Gr}}(1)) \circ \Psi_{\Gr} \cong \Psi_{\Gr} \circ (-\otimes \cO_{Y_{\Gr}}(1)),
\\
(-\otimes\cO_{X_{\Pf}}(1)) \circ \Psi_{\Pf} \cong \Psi_{\Pf} \circ (-\otimes \cO_{Y_{\Pf}}(-1)).
\end{gather*}
\end{prop}

\begin{proof}
Recall from Definition~\ref{def.equivGr} that $\Psi_{\Gr}$ is the composition
\begin{equation*}
\D(Y_{\Gr}) \stackrel{\pi^*}{\longrightarrow} \D(X_{\Gr}|_{Y_{\Gr}}) \stackrel{k_*}{\longrightarrow} \D(X_{\Gr}, f),
\end{equation*}
where morphisms are as follows:
\begin{equation*}
\begin{tikzpicture}[scale=1.5,xscale=1.3]
\node (A) at (0,0) {$X_{\Gr}$};
\node (B) at (-1,0) {$X_{\Gr}|_{Y_{\Gr}}$};
\node (C) at (0,-1) {$\Gr$};
\node (D) at (-1,-1) {$Y_{\Gr}$};
\draw[right hook->] (B) -- node[above] {$ \scriptstyle k$} (A);
\draw[right hook->] (D) -- node[above] {} (C);
\draw[<-] (C) -- node[left] {} (A);
\draw[<-] (D) -- node[left] {$\scriptstyle \pi$} (B);
\end{tikzpicture}
\end{equation*}

Now assume given sheaves $\cL$ on $X_{\Gr}$ and $\cM$ on $Y_{\Gr}$ such that $k^*\cL \cong \pi^*\cM$. Then, using the projection formula, we have isomorphisms as follows:
\begin{align*}
(-\otimes\cL) \circ \Psi_{\Gr} & = (-\otimes\cL) \circ (k_* \pi^*)
 \cong k_* \circ (-\otimes k^* \cL) \circ \pi^* \\
& \cong k_* \circ (-\otimes \pi^* \cM) \circ \pi^*
 \cong (k_* \pi^*)\circ (-\otimes \cM) \\
& \cong \Psi_{\Gr} \circ (-\otimes \cM).
\end{align*}
We immediately deduce the first statement, by taking $\cL=\cO_{X_{\Gr}}(1)$ and $\cM=\cO_{Y_{\Gr}}(1)$.


From Proposition~\ref{prop.Flandsinwindow} we have that $\Psi_{\Pf}^{\underline{m}}$ is isomorphic to the composition
\begin{equation*}
\D(Y_{\Pf}) \xrightarrow{(\pi j)^*} \D(\Gamma) \xrightarrow{(k j)_*} \D(X_{\Pf}, f),
\end{equation*}
where morphisms are as shown below. Here we write $\D(\Gamma)$ rather than $\D(\Gamma, f)$ because $f$ restricts to zero on $\Gamma \subset X_{\Pf}$ by Definition~\ref{defn:Gamma}.
\begin{equation*}
\begin{tikzpicture}[scale=1.5,xscale=1.2]
\node (A) at (0,0) {$X_{\Pf}$};
\node (B) at (-1,0) {$X_{\Pf}|_{Y_{\Pf}}$};
\node (C) at (0,-1) {$\P^6$};
\node (D) at (-1,-1) {$Y_{\Pf}$};
\node (E) at (-2,0) {$\Gamma$};
\draw[right hook->] (B) -- node[above] {$ \scriptstyle k$} (A);
\draw[right hook->] (D) -- node[above] {} (C);
\draw[<-] (C) -- node[left] {} (A);
\draw[<-] (D) -- node[left] {$\scriptstyle \pi$} (B);
\draw[right hook->] (E) -- node[above] {$\scriptstyle j$} (B);
\end{tikzpicture}
\end{equation*}
Applying a similar argument to $\Psi_{\Pf}$ using sheaves $\cL=\cO_{X_{\Pf}}(1)$ and $\cM=\cO_{Y_{\Pf}}(-1)$, we deduce the second statement. In this case it suffices that $k^*\cL \cong \pi^*\cM$, as this gives $(kj)^*\cL \cong (\pi j)^*\cM$ after applying $j^*$. The minus sign here arises as in Remark~\ref{rem.sign}.
\end{proof}

\begin{thm}\label{main thm}
There is an action of the groupoid $\pi_1(\cM,\{m_{\Gr},m_{\Pf}\})$ for $\cM=S^2 - \{\text{$5$ points}\}$ on $\D(Y_{\Gr})$ and $\D(Y_{\Pf})$, given by the following diagram.
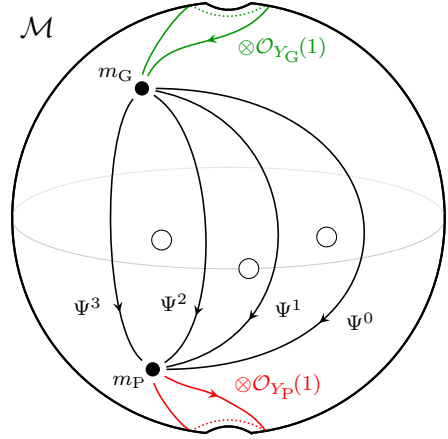
\picskip
\begin{center}
\skmsPic{2}{0}
\end{center}
\end{thm}
\begin{proof} Removing a point from the far side of the sphere, the claim is that the groupoid acts by the following diagram, subject to a relation: that monodromy starting at either $\D(Y_{\Gr})$ or~$\D(Y_{\Pf})$ and then going around a large circle gives the identity up to isomorphism:
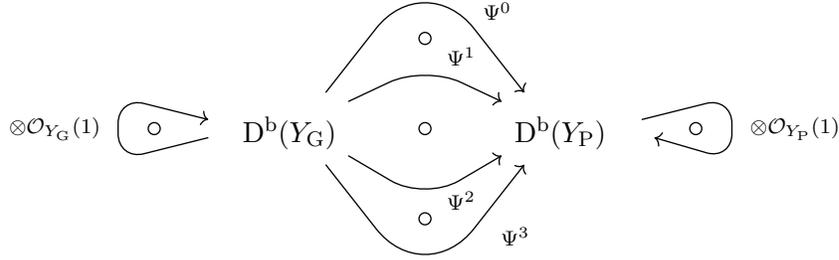
\begin{center}
\begin{tikzpicture}
\node at (0,0) {$\halfFullMonodromiesStandardFlop{$\D(Y_{\Pf})$}{$\D(Y_{\Gr})$}{\Psi}$};
\end{tikzpicture}
\end{center}

\noindent
Now we have that $\cW^3 = \cW^0 \otimes \cO_{\mfX}(-1)$, and note that
\[
i_{\Gr}^* \cO_{\mfX}(m) \cong \cO_{X_{\Gr}}(m) \qquad\text{and}\qquad
i_{\Pf}^* \cO_{\mfX}(m) \cong \cO_{X_{\Pf}}(m).
\]
It then follows easily from Definition \ref{def.windowequiv} of the window equivalences~$\Psi_{\cW}^{\underline{m}}$, of which the $\Psi_{\cW}^l$ are examples, that
\[
\Psi_{\cW}^3 \cong (-\otimes\cO_{X_{\Pf}}(-1))\circ\Psi_{\cW}^0 \circ(-\otimes\cO_{X_{\Gr}}(1)),
\]
noting the opposite signs in the line bundle twists. Combining with the above Proposition~\ref{prop.cyinter} then gives an isomorphism
\[
\Psi^3 \cong (-\otimes\cO_{Y_{\Pf}}(1))\circ\Psi^0 \circ(-\otimes\cO_{Y_{\Gr}}(1)),
\]
which yields the required relation, and completes the proof.\end{proof}

\section{Monodromy action}\label{sec twists}

In this section I describe the differences between the equivalences~$\Psi^l$ in Theorem~\ref{keythm.groupoid} (Theorem~\ref{main thm} above) as spherical twists on the Calabi--Yau $Y_{\Gr}$, and thereby prove Theorem~\ref{keythm.twists}.

To describe the difference $\Psi'^{-1} \circ \Psi$ between such equivalences $\Psi$ and~$\Psi'$, we will use a~functor~$\Tr$ that ``transfers'' between the two corresponding windows $\cW$ and~$\cW'$. The difference is then described as a twist functor, denoted $\Tw$. The following proposition gives the formal part of this procedure.

For convenience, take two functors from $\cW = \cW^{\underline{m}}$ as follows:
\begin{equation*}
\Phi_{\Gr} = (\Psi_{\Gr})^{-1} \circ i_{\Gr}^*, \qquad
\Phi_{\Pf} = (\Psi^{\underline{m}}_{\Pf})^{-1} \circ i_{\Pf}^*
\end{equation*}
so that the factorization of $\Psi = \Psi^{\underline{m}}$ via $\cW$ takes the form
\begin{equation*}
\Psi \cong \Phi_{\Pf} \circ (\Phi_{\Gr})^{-1} \colon\ \D(Y_{\Gr}) \to \cW \to \D(Y_{\Pf}).
\end{equation*}

\begin{prop}[{\cite[Proposition 2.2]{DS1}}]\label{prop.wshift}
For windows $\cW$ and $\cW'$, assume given a functor $\Tr\colon \cW \to \cW'$ that intertwines with an autoequivalence $\Tw$ of $\D(Y_{\Gr})$ and with the identity on $\D(Y_{\Pf})$, namely
\begin{equation}\label{eqn.intertwine}
\Phi_{\Gr} \circ \Tr \cong \Tw \circ \,\Phi_{\Gr}
\qquad\text{and}\qquad
\Phi'_{\Pf} \circ \Tr \cong \Phi_{\Pf}.
\end{equation}
Then there is an isomorphism
\begin{equation*}
\Psi'^{-1} \circ \Psi \cong \Tw,
\end{equation*}
where $\Psi$ and $\Psi'$ are the window equivalences associated to $\cW$ and $\cW'$.
\end{prop}
\begin{proof}
This is a diagram chase of functors. In the reference $\Tr$ was assumed to be a restriction of an endofunctor of the category containing $\cW$, in our case $\D(\mfX,f)$, but the proof proceeds without this assumption.
\end{proof}

To construct functors $\Tr$ we define endofunctors of $\D(\mfX,f)$, and then show that they restrict to functors between windows. Recall that~$\mfX$ is a vector bundle over a stack
\begin{equation*}
\mfG = \quotstack{\Hom(S,V)}{\GL(S)}
\end{equation*}
and write $j\colon \mfG \to \mfX$ for the inclusion of the zero section. Recall also that
\[
j_! = j_*(\omega_j[\dim j]\otimes - )\cong j_*(\cO_{\mfG}(-7)[-7] \otimes -),
\]
where we use that $\omega_j$ is the determinant of the normal bundle of $\cG$.

\begin{defn} Take endofunctors of $\D(\mfX,f)$
\begin{equation*}
\Tr^0 = \Tr(j_! \cO_{\mfG}), \qquad
\Tr^1 = \Tr(j_! S), \qquad
\Tr^2 = \Tr\big(j_! S^{2}\big),
\end{equation*}
where
\begin{equation*}
\Tr(\cE)=\Cone\big({\cE} \otimes \Hom_{\mfX}({\cE},-) \to \id \big).
\end{equation*}
Here, and elsewhere in this section, we again put $S^2$ for $\Sym^2 S$.
\end{defn}

\begin{rem}The functorial cone above is shorthand for the usual Fourier--Mukai constructions.
\end{rem}

I then claim the following.

\begin{prop} The $\Tr^l$ restrict to functors $\Tr^l\colon \cW^l \to \cW^{l+1}.$
\end{prop}

\begin{proof}
I first show this for $\Tr^0$, as the others are similar. I claim that the functor acts as the identity on all generators of the window $\cW^0$ except~$\cO$. First observe that
\begin{align}\label{eq.twisthom} \Hom_{\mfX}(j_!\cO_{\mfG},-)
& \cong \Hom_{\cG}(\cO_{\cG},j^* -).
\end{align}
This is zero except on $\cO$ by inspection of Collection~\ref{cola}, noting the vanishing in Corollary~\ref{cor.Kuznetsov_extended_new_mod}. It~also implies, using that objects in the collection are exceptional, that \begin{equation*}\Hom_{\mfX}(j_!\cO_{\mfG},\cO) = \C.\end{equation*}
A generator of this $\Hom$ may be seen explicitly by writing down the Koszul resolution of~$j_!\cO_{\cG}$, namely
\begin{equation}\label{eq.koszul}
\cO \to \cO(-1)^{\oplus 7} \to \dots \to \cO(-6)^{\oplus 7} \to \cO(-7).
\end{equation}
We thus see that $\Tr^0 (\cO) = \Cone ( j_!\cO_{\cG} \to \cO )$ is quasi-isomorphic to a complex given by
\begin{equation*} \cO(-1)^{\oplus 7} \to \dots \to \cO(-6)^{\oplus 7} \to \cO(-7) \end{equation*}
which, in particular, lies in the window $\cW^1$ given as Collection~\ref{colb}.


For the other two functors $\Tr^l$ for $l=1,2$, we replace \eqref{eq.twisthom} with
\begin{align*} \Hom_{\mfX}\!\big(j_! S^{l},-\big) & \cong \Hom_{\cG}\!\big(S^{l}, j^*\!-\!\big)
\end{align*}
and repeat the same argument, where we tensor
 \eqref{eq.koszul} by $S^{l}$.
\end{proof}

I now complete the proof of the assumptions of Proposition~\ref{prop.wshift}, in particular property \eqref{eqn.intertwine}, after constructing autoequivalences of $\D(Y_{\Gr})$ as follows.

\begin{defn} Take endofunctors of $\D(Y_{\Gr})$
\begin{equation*}
\Tw^0 = \Tw(\cO_{Y_{\Gr}}), \qquad
\Tw^1 = \Tw(S_{Y_{\Gr}}), \qquad
\Tw^2 = \Tw\big(S^{2}_{Y_{\Gr}}\big),
\end{equation*}
where
\begin{equation*}
\Tw(\cF)=\Cone\big(\cF \otimes \Hom_{Y_{\Gr}}(\cF,-) \to \id \big).
\end{equation*}
\end{defn}

\begin{prop}\label{prop.sphobj}
Each of the $\Tw^l $ is a twist by a spherical object, and therefore is an auto\-equi\-valence.
\end{prop}
\begin{proof}
To see this for $\Tw^0$, write $k\colon Y_{\Gr} \into \Gr$ and note that
\begin{equation*}
\Ext^\bullet_{Y_{\Gr}}(\cO_{Y_{\Gr}},\cO_{Y_{\Gr}}) = \Ext^\bullet_{Y_{\Gr}}(k^*\cO_{\Gr},k^*\cO_{\Gr}) \cong \Ext^\bullet_{{\Gr}}(\cO_{\Gr},k_* k^*\cO_{\Gr}).
\end{equation*}
After taking a Koszul resolution of $k_* k^*\cO_{\Gr}$, this may be calculated by a spectral sequence from $\Ext^\bullet_{\Gr}(\cO_{\Gr}, \wedge^\bullet \cO_{\Gr}(-1)^{\oplus 7})$. By examination of the exceptional collectional given as Collection~\ref{cola}, the only non-trivial terms occuring are $\Ext^0(\cO_{\Gr}, \cO_{\Gr})\cong\C$ and
\begin{equation*}
\Ext^\bullet_{\Gr}(\cO_{\Gr}, \cO_{\Gr}(-7)) \cong H^{\bullet}_{\Gr}(\cO_{\Gr}(-7))\cong\C[-\dim \Gr]
\end{equation*}
by duality. Noting that $Y_{\Gr}$ is Calabi--Yau, the result follows. A similar argument applies to the other $\Tw^l$.
\end{proof}

\begin{proof}[Proof of intertwinement~\eqref{eqn.intertwine}]

The isomorphism $\Phi'_{\Pf} \circ \Tr^0 \cong \Phi_{\Pf}$ follows immediately from the cone construction of $\Tr^0$: recall that $\Phi_{\Pf} = (\Psi^{\underline{m}}_{\Pf})^{-1} \circ i_{\Pf}^*$ and note that $i_{\Pf}^* \,j_!\,\cO_{\cG} = 0$ for support reasons. The same argument applies to the other $\Tr^l$.


The other intertwinements
\begin{equation}\label{eq.intwn}\Phi_{\Gr} \circ \Tr^l \cong \Tw^l \circ \,\Phi_{\Gr}\end{equation} are more involved, but essentially follow from the exceptional collection property. We give the argument for $\Tr^0$ before explaining how it adapts to the other $\Tr^l$.

As preparation,
recall that $\Phi_{\Gr} = \Psi_{\Gr}^{-1} \circ i_{\Gr}^*$, and write $i$ in place of $i_{\Gr}$ for clarity. It will be more convenient to prove the following, which is equivalent to~\eqref{eq.intwn}.
\begin{equation}\label{eq.eqint}
i^* \circ \Tr^l \cong \Psi_{\Gr} \circ \Tw^l \circ \Psi_{\Gr}^{-1} \circ \,i^*.
\end{equation}
By standard facts about spherical twists $\Psi_{\Gr} \circ \Tw^0 \circ \Psi_{\Gr}^{-1}$ is given by the spherical twist around $\Psi_{\Gr} (\cO_{Y_{\Gr}})$. By Proposition~\ref{prop.images} this is isomorphic to~$ j_! \cO_{\Gr}$, and therefore
\begin{equation*}\Psi_{\Gr} \circ \Tw^0 \circ \Psi_{\Gr}^{-1} \cong \Cone(j_!\cO_{\Gr} \otimes \Hom_{X_{\Gr}}(j_!\cO_{\Gr}, -) \to \id ).\end{equation*}

Now we have a fibre product diagram
\begin{equation*}
\begin{tikzpicture}[scale=1.4]
\node (A) at (0,0) {$\mfX$};
\node (B) at (-1,0) {$X_{\Gr}$};
\node (C) at (0,-1) {$\mfG$};
\node (D) at (-1,-1) {$\Gr$};
\draw[right hook->] (B) -- node[above] {$ \scriptstyle i$} (A);
\draw[right hook->] (D) -- node[below] {$ \scriptstyle i$} (C);
\draw[right hook->] (C) -- node[left] {$\scriptstyle j$} (A);
\draw[right hook->] (D) -- node[left] {$\scriptstyle j$} (B);
\end{tikzpicture}
\end{equation*}
and noting that the $i$ are open immersions and therefore flat, we find by base change that \begin{equation*}j_! \cO_{\Gr} = j_! i^* \cO_{\cG} \cong i^* j_! \cO_{\cG}.\end{equation*}
We thence have isomorphisms of functors $\D(\mfX,f)\to\D(X_{\Gr},f)$ as follows:
\begin{gather*}
i^* \circ \Tr^0 \cong \Cone(i^* j_!\cO_{\cG} \otimes \Hom_{\mfX}(j_!\cO_{\cG},-) \to i^*) ,
\\
\Psi_{\Gr} \circ \Tw^0 \circ \Psi_{\Gr}^{-1} \circ \,i^* \cong \Cone(i^* j_!\cO_{\cG} \otimes \Hom_{X_{\Gr}}(i^*j_!\cO_{\cG},i^* -) \to i^* ).
\end{gather*}

Comparing the two cones above suggests that the intertwinement may follow from an isomorphism
\begin{equation}\label{eq.keyiso}
\Hom_{\mfX}(j_!\cO_{\cG},-) \cong \Hom_{X_{\Gr}}(i^*j_!\cO_{\cG},i^* -)\text{ on }\cW^0.
\end{equation}
This isomorphism does indeed hold: we follow an argument for this, which furthermore gi\-ves~\eqref{eq.eqint}, in \cite[proof of Lemma 3.17, end of Section 3.2.2]{DS1}. (Note that the setting there is more general, involving a spherical functor, not just a spherical object.) According to the argument, it suf\-fices if the natural morphism of functors from $\D(\cG)$
\begin{equation*}
\tau\colon\ R\Gamma_{\cG} \to R\Gamma_{\Gr} \circ i^*
\end{equation*}
is an isomorphism on the subcategory $j^* \cW^0$ of $\D(\cG)$. To see why we take this subcategory, note that $j^*\colon \D(\mfX,f) \to \D(\cG)$ because $f|_{\cG}=0$, and the left-hand side of~\eqref{eq.keyiso} is given by
\begin{gather}
\Hom_{\mfX}(j_!\cO_{\cG},-) \cong \Hom_{\cG}(\cO_{\cG},j^* -) \cong R\Gamma_{\cG} (j^* -).
\label{eq.homtwisted}
\end{gather}

Now we have that $j^* - \cong \hom_{\cG} (\cO, j^* - ).$ It therefore suffices to check if $\tau$ is an isomorphism on objects
\begin{equation*}
\cA = \hom_{\cG} (\cO, j^* \cB ),
\end{equation*}
where $\cB$ is a generator of $\cW^0$.

We determine $\tau_{\cA}$. Firstly, $R^{>0}\Gamma_{\cG} (\cA) = 0$ because $\cA$ is a vector bundle on an affine stack. Furthermore, $R^{>0}\Gamma_{\Gr} (i^* \cA)
\cong \Ext^{>0}_{\Gr}(\cO,i^*j^*\cB)$.
But then this vanishes by the exceptional collection property, as $\cO$ is on the right-hand side of the collection associated to $\cW^{0}$, namely Collection~\ref{cola}. Finally, we have that
$H^0 \tau_\cA$ is an isomorphism by normality of~$\cG$, as the codimension of $\cG - \Gr$ is greater than two. We deduce that $\tau_\cA$ is an isomorphism, and thence $\tau$ is an isomorphism, thereby proving the intertwinement~\eqref{eq.intwn} for~$\Tr^0$.

For the other transfer functors $\Tr^1$ and $\Tr^2$, it now suffices that $\tau$ is an isomorphism on~$S^\vee \otimes j^* \cW^1$ and $S^{\vee 2} \otimes j^* \cW^2 $, respectively. These categories appear because the role
of~\eqref{eq.homtwisted} is repla\-ced~by
\begin{gather*}
\Hom_{\mfX}(j_! S,-) \cong R\Gamma_{\cG}\big(S^\vee \otimes j^* (-)\big),\\
\Hom_{\mfX}\big(j_! S^{2},{-}\big) \cong R\Gamma_{\cG}\big(S^{\vee 2} \otimes j^* (-)\big).
\end{gather*}
But then the above argument suffices, using that $S$ is on the right-hand side for Collection~\ref{colb}, and $S^{2}$ is on the right-hand side for Collection~\ref{colc}, and that Proposition~\ref{prop.images} continues to hold with occurences of~$\cO$ replaced with $S$ or~$S^{2}$.\end{proof}

Combining the above, and applying Proposition~\ref{prop.wshift}, we can describe differences between equivalences as follows.

\begin{prop}\label{prop.twists} The twists $\Tw$\,correspond to ``window shifts'' as follows:
\begin{gather*}
 \big(\Psi^1\big)^{-1} \circ \Psi^0 \cong \Tw^0 = \Tw(\cO_{Y_{\Gr}}), \\
\big(\Psi^2\big)^{-1} \circ \Psi^1 \cong \Tw^1 = \Tw(S_{Y_{\Gr}}), \\
\big(\Psi^3\big)^{-1} \circ \Psi^2 \cong \Tw^2 = \Tw(S^{2}_{Y_{\Gr}}).
\end{gather*}
\end{prop}

We may then complete the proof of Theorem~\ref{keythm.twists}, as follows.

\begin{thm}\label{thm.twists}
There is an action of the fundamental group $\pi_1(\cM,m_{\Gr})$ on $\D(Y_{\Gr})$ given by the following diagram, with $\Psi=\Psi^3{:}$
\picskip
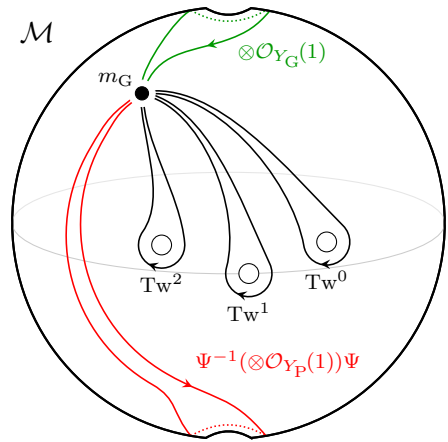
\begin{center}
\skmsPic{1}{1}
\end{center}
\end{thm}
\begin{proof}We take the groupoid action of Theorem~\ref{keythm.groupoid} (Theorem~\ref{main thm}) and forget one of the basepoints $m_{\Pf}$, as follows. Restricting that action to a chart containing both basepoints, we obtain the left-hand picture below. Proposition~\ref {prop.twists} then says that monodromies at $\D(Y_{\Gr})$ around equatorial holes are given by the right-hand picture:
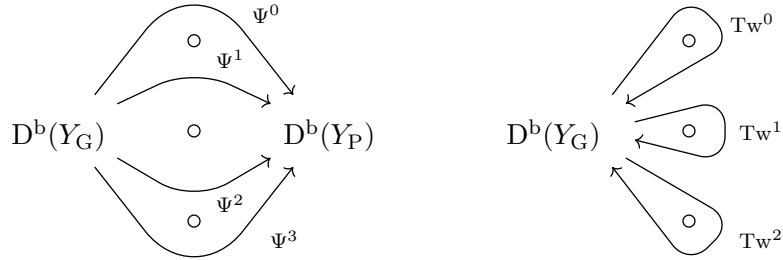
\begin{center}
\begin{tikzpicture}
\node at (0,0) {$\monodromiesStandardFlop{$\D(Y_{\Gr})$}{$\Tw^{2}$}{$\Tw^{1}$}{$\Tw^{0}$}$};
\node at (-6,0) {$\halfMonodromiesStandardFlop{$\D(Y_{\Pf})$}{$\D(Y_{\Gr})$}{\Psi}$};
\end{tikzpicture}
\end{center}
The remaining two monodromies are immediate, giving the result.
\end{proof}

\begin{rem} It is natural to ask for a Pfaffian analogue of Theorem~\ref{thm.twists} with a basepoint $m_{\Pf}$ in place of $m_{\Gr}$. This seems an interesting and approachable problem, but it appears that the method here would need modification to solve it, as follows.

A first step could be to establish spherical objects on the $3$-fold $Y_{\Pf}$ and prove an analogue of~Proposition~\ref{prop.images}, giving their images under appropriate equivalences $\Psi_{\Pf}^{\underline{m}}$. To continue the strategy above, we would then need manageable resolutions for the images, as in the Koszul resolution of $j_!\cO_{\cG}$ in\opt{a}{~}\opt{s}{ }\eqref{eq.koszul}. These should be obtainable, but can be expected to be more complicated than Koszul resolutions, because of the Eagon--Northcott resolutions in the construction of $\Psi_{\Pf}^{\underline{m}}$ from the proof of Proposition~\ref{prop.Flandsinwindow}.
\end{rem}

\begin{rem} I do not discuss whether the action of Theorem~\ref{thm.twists} is faithful. If it is, then the spherical twists by the objects in Proposition~\ref {prop.twists} act as a free group: it could be interesting to try to prove this by some B-side analogue of results of Keating on free group actions by symplectic Dehn twists~\cite{Keating}.
\end{rem}

\subsection*{Acknowledgements}
In celebration of his 77th birthday, I am pleased to express my gratitude to Kyoji Saito for his kindness and interest over the years. I also want to thank N.~Addington and E.~Segal for the great experience of working on our paper~\cite{ADS}, of which this work is a continuation.

I am supported by the Yau MSC, Tsinghua University, and the Thousand Talents Plan. I also acknowledge the support of WPI Initiative, MEXT, Japan, and of EPSRC Programme Grant EP/R034826/1.

I am grateful for conversations with K.~Hori, who made me aware of the windows used here, for discussions with
R.~Eager, 
A.~Kuznetsov, 
and M.~Romo, 
and for helpful suggestions from anonymous referees.
In the last stage of this project, I was away from my home institution due to the coronavirus pandemic, so I am especially appreciative of hospitality and support from Kavli IPMU, in particular Y.~Ito and M.~Kapranov, and from B.~Kim at KIAS, and M.~Wemyss at Glasgow University.

\pdfbookmark[1]{References}{ref}
\LastPageEnding

\end{document}